\documentclass[12pt]{amsart}
\usepackage{amsmath, amsthm, amssymb, amsfonts}
\usepackage{mathtools} 
\usepackage{diagbox} 
\usepackage{array}
\usepackage{float}
\usepackage{thmtools}
\usepackage{thm-restate}
\usepackage{tikz-cd}
\usepackage{enumitem}
\usepackage{etoolbox}
\usepackage{ifthen}
\usepackage[backref,backend=biber,style=alphabetic,maxbibnames=10,sorting=nyt,giveninits]{biblatex}
\usepackage[pdfusetitle]{hyperref}
\usepackage[nameinlink]{cleveref}
\usepackage{microtype,fullpage} 

\usepackage{xpatch}
\DeclareFieldFormat*{issn}{}
\DeclareFieldFormat*{doi}{}
\DeclareFieldFormat*{url}{}
\DeclareFieldFormat[unpublished]{url}{\mkbibacro{URL}\addcolon\space\url{#1}}
\DeclareFieldFormat[online]{url}{\mkbibacro{URL}\addcolon\space\url{#1}}
\AtEveryBibitem{\ifentrytype{book}{\clearfield{pages}}{}}
\AtEveryBibitem{\ifentrytype{book}{\clearfield{isbn}}{}}
\xpatchbibdriver{article}{\setunit{\addspace}\newblock
 \iftoggle{bbx:related}}{\newunit\newblock
 \iftoggle{bbx:related}}{}{}

\declaretheorem[name=Theorem,numberwithin=section]{thm}
\newtheorem{theorem}[thm]{Theorem}
\newtheorem{lemma}[thm]{Lemma}
\newtheorem{corollary}[thm]{Corollary}

\newtheorem{prop}[thm]{Proposition}

\newtheorem*{theorem*}{Theorem}

\theoremstyle{definition} 
\newtheorem{definition}[thm]{Definition}

\newtheorem{question}[thm]{Question}

\theoremstyle{remark}  
\newtheorem{remark}[thm]{Remark}
\newtheorem{example}[thm]{Example}
\newtheorem{conj}[thm]{Conjecture}
\newtheorem{conjecture}[thm]{Conjecture}

\newtheorem{claim}{Claim}

\newcommand{\define}[4]{\expandafter#1\csname#3#4\endcsname{#2{#4}}}
\forcsvlist{\define{\DeclareMathOperator}{}{}}{im,coker,rad,nil,Ann,Ass,codim,Spec,mSpec,diam,ord,Supp,supp,disc,vol,rank,Sym,Ind, Pic, Proj, Alb, rk}
\forcsvlist{\define{\newcommand}{\mathrm}{}}{Hom,Mor,id,GL,SL,SO,SU,M,Mat,Ext,Tor,Res,Cor,Inf,End,Irr,Aut,Gal,lcm,tr,sign,triv,diag,Map,op,ev,act,alg,unr,nr,ab,Div, Ch, gr}
\forcsvlist{\define{\newcommand}{\mathbb}{}}{N,Z,Q,R,C,F,G,T,A,B,D,K}
\forcsvlist{\define{\newcommand}{\mathfrak}{f}}{q,p,m,n}
\newcommand{\cO}{\mathcal{O}}

\newcommand{\iso}{\cong}

\DeclareMathOperator*{\tensor}{\otimes}

\renewcommand{\Mat}{\mathrm{M}}
\newcommand{\initial}[1]{\mathrm{in}(#1)}
\DeclareMathOperator{\ch}{char}
\DeclarePairedDelimiter\floor{\lfloor}{\rfloor}
\DeclareMathOperator{\height}{ht}

\RequirePackage{xparse}
\ifx \set \undefined%
  \NewDocumentCommand\set{mg}{%
    \ensuremath{\left\{ #1 \IfNoValueTF{#2}{}{\:\middle|\: #2} \right\}}}%
\else%
  \RenewDocumentCommand\set{mg}{%
    \ensuremath{\left\{ #1 \IfNoValueTF{#2}{}{\:\middle|\: #2} \right\}}}%
\fi

\newcommand{\maps}[3][n]{%
  \ifthenelse{\equal{#1}{l}}{\,{:}\,#2\!&\to\!#3}{}%
  \ifthenelse{\equal{#1}{r}}{\,{:}\,#2\!\to&\!#3}{}%
  \ifthenelse{\equal{#1}{n}}{\,{:}\,#2\!\to\!#3}{}%
}
\newcommand{\lmaps}[3][n]{%
  \ifthenelse{\equal{#1}{l}}{\,{:}\,#2\!&\longrightarrow\!#3}{}%
  \ifthenelse{\equal{#1}{r}}{\,{:}\,#2\!\longrightarrow&\!#3}{}%
  \ifthenelse{\equal{#1}{n}}{\,{:}\,#2\!\longrightarrow\!#3}{}%
}
\newcommand{\lmapsto}[3][n]{%
  \ifthenelse{\equal{#1}{l}}{#2\!&\longmapsto\!#3}{}%
  \ifthenelse{\equal{#1}{r}}{#2\!\longmapsto&\!#3}{}%
  \ifthenelse{\equal{#1}{n}}{#2\!\longmapsto\!#3}{}%
}

\newcommand{\widebar}{\overline}

\title{On Trace Zero Matrices and Commutators}
\author{Makoto Suwama}
\address{Department of Mathematics\\ University of Georgia\\
 Athens, Georgia 30602, USA}
\email{makoto.suwama@gmail.com}

\begin{document}

\begin{abstract}
  Given any commutative ring $R$, a commutator of two $n\times n$ matrices over $R$ has trace $0$. In this paper, we study the converse: whether every $n \times n$ trace $0$ matrix is a commutator. We show that if $R$ is a B\'{e}zout domain with algebraically closed quotient field, then every $n\times n$ trace $0$ matrix is a commutator. We also show that if $R$ is a regular ring with large enough Krull dimension relative to $n$, then there exist a $n\times n$ trace $0$ matrix that is not a commutator. This improves on a result of Lissner by increasing the size of the matrix allowed for a fixed $R$. We also give an example of a Noetherian dimension $1$ commutative domain $R$ that admits a $n\times n$ trace $0$ non-commutator for any $n\ge 2$.
\end{abstract}
\maketitle

\section{Introduction}
Let $R$ be a commutative ring. Given two $n\times n$ matrices $A$ and $B$ in $\Mat_n(R)$, recall that the commutator of $A$ and $B$ is denoted by $[A,B]:=AB-BA$. It is a standard fact that $\tr(AB) = \tr(BA)$, and since the trace is additive, it follows that $\tr([A,B]) = \tr(AB-BA)=0$. So it is natural to wonder if the converse is also true. That is, given an $n\times n$ trace $0$ matrix $C$, do there exist two $n \times n$ matrices $A$ and $B$ such that $C=[A,B]$?

The answer to the above question depends on the underlying commutative ring $R$ where the entries lie and the size $n$ of the matrix. If $R$ is a field, then every $n\times n$ trace $0$ matrix is a commutator. This was proven by Shoda in \cite{Shoda37} for a characteristic $0$ field and by Albert and Muckenhoupt in \cite{AM57} for a field with any characteristic.

More recently, Laffey and Reams showed that for every $n\ge 1$, any $n\times n$ trace $0$ matrix is a commutator over $R=\Z$ in \cite{LR94}, and Stasinski generalised it to an arbitrary principal ideal ring in \cite{Stasinski16}. Stasinski subsequently showed in \cite{Stasinski18} that over principal ideal rings, every trace $0$ matrix is a commutator of trace $0$ matrices as well. In this paper, we give a new class of rings where every trace $0$ matrix is a commutator.
\begin{restatable*}{thm}{bezout}
\label{thm:bezoutcommutator}
  Let $R$ be a B\'{e}zout domain with algebraically closed quotient field. Then every trace $0$ matrix in $\Mat_n(R)$ is a commutator for any $n\ge 1$.   
\end{restatable*}
In \Cref{sec:hollow}, we also discuss hollow matrices and nilpotent matrices which are both special cases of trace $0$ matrices, and prove the following theorem.
\begin{restatable*}{thm}{dedekind}
\label{thm:dedekindalgebra}
  Let $R$ be a Pr\"{u}fer domain that is also a $k$-algebra over an infinite field $k$. Then for any $n\ge 1$, every nilpotent matrix in $\Mat_n(R)$ is a commutator.  
\end{restatable*}
This is a partial progress towards answering whether every $n\times n$ trace $0$ matrix over a Dedekind domain is a commutator for any $n \ge 1$. No results for Dedekind domains were previously known for the case when $n\ge 3$ and $R$ is not a principal ideal domain.

On the other hand, there are rings where there is a trace $0$ matrix that is not a commutator. Lissner in \cite{Lissner61} showed that if $R$ is a polynomial ring in $m\ge 3$ variables, then there is an $n\times n$ trace $0$ non-commutator for any $2\le n\le \tfrac{m+1}{2}$. Mesyan in \cite{Mesyan06} used a similar idea to create a $2\times 2$ trace $0$ non-commutator over a ring with a maximal ideal satisfying certain properties. In \Cref{sec:two} and \Cref{sec:combinatorics}, we extend the above works of Lissner and Mesyan to construct a trace $0$ non-commutator over a more general class of rings.

\begin{restatable*}{thm}{main}
  \label{thm:main}
  Let $R$ be a commutative ring with a maximal ideal $\fm\subset R$ such that $R_\fm$ is a regular local ring of dimension $m\ge 3$. Suppose further that $n\ge 2$ is an integer such that $n \le \tfrac{m^2+2m+5}{8}$.
  Then there exists a trace $0$ non-commutator in $M_n(R)$.
\end{restatable*}
In particular, we improve the bound on the size of the matrix where we can produce a trace $0$ non-commutator for a fixed ring; If $m$ is the dimension of the ring and $n$ is the size of the matrix, then Lissner requires $n\le \tfrac{m+1}{2}$, so we have improved the bound on the size of the matrix from linear to quadratic in $m$. This improvement comes from solving a certain combinatorial problem which can be thought of either as a packing problem or a graph theory problem. The construction of the matrix will be described in \Cref{sec:two}, while the combinatorics will be explained in \Cref{sec:combinatorics}.

In \Cref{sec:ring}, we give an example of a Noetherian dimension $1$ ring where there is an $n\times n$ trace $0$ non-commutator for any $n\ge 2$.
\begin{restatable*}{thm}{unbounded}
  \label{thm:unbounded}
  There exists a Noetherian commutative domain $\Lambda$ with dimension $1$ such that for every $n\ge 2$, there exists a trace $0$ non-commutator in $\Mat_n(\Lambda)$.
\end{restatable*}

Finally, in \Cref{sec:2x2}, we discuss $2\times 2$ matrices. The $2\times 2$ matrices are special since every $2\times 2$ trace $0$ matrix being a commutator is equivalent to every vector in $R^3$ being a cross product. This will be discussed further in \Cref{sec:2x2} along with \Cref{thm:characterisation} which characterises rings where every $2\times 2$ trace $0$ matrix is a commutator. After the characterisation, we prove the following theorem.
\begin{restatable*}{thm}{fpalgebra}
  \label{cor:fpalgebraisop}
  Let $R$ be a regular finitely generated $\widebar{\F}_p$-algebra of dimension $2$. Then $R$ is an OP-ring and every trace $0$ matrix in $\Mat_2(R)$ is a commutator.
\end{restatable*}
The analogue of the above theorem for $\widebar{\Q}$-algebras is also true if the Bloch-Beilinson conjecture holds (see \Cref{thm:bbopring}).

\section{B\'{e}zout Domains and Pr\"{u}fer Domains}\label{sec:hollow}
Recall that a \emph{B\'{e}zout domain} is a domain where every finitely generated ideal is principal and a \emph{Pr\"{u}fer domain} is a domain where every finitely generated ideal is invertible. They are non-Noetherian analogues of a PID and a Dedekind domain respectively. A Noetherian B\'{e}zout domain is a PID and a Noetherian Pr\"{u}fer domain is a Dedekind domain. Note that a B\'{e}zout domain is a Pr\"{u}fer domain since principal ideals are always invertible. For a reference on B\'{e}zout domains and Pr\"{u}fer domains see \cite[Chapter III]{fl01}. 

The main idea in this section is about finding an appropriate basis of $R^n$ which puts a given matrix $A$ in a form that makes it easier to show that $A$ is a commutator. We consider $A$ as a matrix over the quotient field $K$ of $R$ to find an appropriate filtration of $K^n$ and then bring down the filtration to $R^n$ to find the new basis of $R^n$. We first prove some lemmas related to bringing down a $K$-vector space filtration to an $R$-module filtration.

Recall that given an $R$-module $M$, there is a natural map $M\to M\tensor_R K$ sending $m\in M$ to $m\tensor 1\in M\tensor_R K$. If $M$ is torsion-free, then the map is an injection, so we may consider $M\subset M\tensor_R K$ in such a case.
\begin{lemma}\label{lem:torsionfree}
  Let $R$ be a domain, $K$ be the quotient field of $R$ and $M$ be a torsion-free $R$-module. If $U\subset V\subset M\tensor_R K$ are $K$-subspaces, then $(V\cap M)/(U\cap M)$ is a torsion-free $R$-module.
\end{lemma}
\begin{proof}
  Let $v\in V\cap M$ and $r\in R$ non-zero such that $rv\in U\cap M$. Then $v\in U$ since $r\ne 0$, and so $v\in U\cap M$. Hence $(V\cap M)/(U\cap M)$ is torsion-free.
\end{proof}

\begin{lemma}\label{lem:finitemodules}
  Let $R$ be a Pr\"{u}fer domain and $K$ be its quotient field. Suppose that $M$ is a finitely generated torsion-free $R$-module and $U\subset V:=M\tensor_R K$ is a $K$-subspace. Then $U\cap M\subset V$ is a finitely generated $R$-module.
\end{lemma}
\begin{proof}
  We will prove that $U\cap M$ is finitely generated by inducting on $n:=\dim_KV$. If $n=1$, then $U=\set{0}$ or $K$, so $U\cap M=\set{0}$ or $M$, and hence $U\cap M$ is finitely generated. 

Suppose it is true for $n-1$. If $U\cap M=\set{0}$, then we are done, so suppose $U\cap M\ne \set{0}$. Then there exists a non-zero $v\in U\cap M$. Consider the quotient map \[q:V\to V/Kv.\] 
Since $M$ is finitely generated, $q(M)\iso M/(Kv\cap M)$ is also finitely generated, and by  \Cref{lem:torsionfree}, it is torsion-free as well. Moreover, $M\tensor_RK=V$, so $q(M)\tensor_RK = q(V)\iso K^{n-1}$. We claim that 
\[q(U)\cap q(M)=q(U\cap M).\]
Since $q(U)\cap q(M)\supset q(U\cap M)$ is clear, we only show that $q(U)\cap q(M)\subset q(U\cap M)$. So let $x\in q(U)\cap q(M)$. Then there exist $u\in U$ and $m\in M$ such that $x=q(u)=q(m)$. Now $m-u\in \ker q=Kv \subset U$, so $m= u + (m-u) \in U$ as well, hence $x=q(m)\in q(U\cap M)$. So we have a finitely generated torsion-free module $q(M)$ with $\dim_Kq(V)=n-1$, and so by the inductive hypothesis, $q(M)\cap q(U)=q(U\cap M)$ is a finitely generated $R$-module.
Hence to show that $U\cap M$ is finitely generated, we only need to show that $\ker q\cap (U\cap M) = Kv\cap M$ is finitely generated. Since $R$ is a Pr\"{u}fer domain, by \cite[Theorem V.2.7]{fl01}, there exist finitely generated ideals $I_1,\dots,I_r\subset R$ such that 
\[M\iso I_1\oplus \dots\oplus I_r.\]
Suppose $(v_1,\dots,v_r)\in \oplus_{i=1}^r I_i$ is the image of $v\in M$ under the isomorphism. Then
\[ Kv\cap M \iso \set{a\in K}{av\in M}= \set{a\in K}{av_i\in I_i\; \forall i=1,\dots r} =   \bigcap_{i=1}^r (I_i: v_i), \]
where $(I_i:v_i):=\set{a\in K}{av_i\in I_i}$ is the ideal quotient. If $v_i\ne 0$, then $(I_i:v_i)=v_i^{-1}I_i$, so $(I_i:v_i)$ is finitely generated since $I_i$ is finitely generated. If $v_i=0$, then $(I_i:v_i)=K$. We assumed $v\ne 0$, so at least one $v_i\ne 0$, and so the intersection
\[\bigcap_{i=1}^r (I_i: v_i) = \bigcap^r_{\substack{i=1 \\ v_i\ne 0}} (I_i: v_i) \]
is non-trivial. Hence $Kv\cap M$ is isomorphic to a finite intersection of finitely generated fractional ideals of a Pr\"{u}fer domain, so it is finitely generated by \cite[Ex. III.1.1]{fl01} and so we are done.
\end{proof}

\begin{theorem}\label{thm:bezouttriangularisation}
  Let $R$ be a B\'{e}zout domain and $K$ be its quotient field. If the characteristic polynomial of a matrix $A\in\Mat_n(R)$ splits completely over $K$, then there exists a basis of $R^n$ such that $A$ is upper triangular with respect to this new basis.
\end{theorem}
\begin{proof}
  If we consider $A\in\Mat_n(K)$, then since the characteristic polynomial splits completely, there exists a filtration of $K$-vector spaces
\[0=V_0\subset V_1\subset\cdots\subset V_n=K^n, \]
such that $\dim_KV_i=i$ and $AV_i\subset V_i$ for all $i=1,\dots, n$ (see e.g. Jordan canonical form in \cite[12.3]{DS04}). If we let $W_i:=V_i\cap R^n$ for $i=0,1,\dots, n$, then we obtain the following filtration of $R$-modules
\[ 0=W_0\subset W_1\subset\cdots \subset W_n=R^n,\]
such that $AW_i\subset W_i$ for all $i=1,\dots, n$. By \Cref{lem:finitemodules}, $W_i$'s are all finitely generated, hence $W_i/W_{i-1}$'s are also finitely generated. Now by \Cref{lem:torsionfree}, $W_i/W_{i-1}$ is also torsion-free, and so it is a finitely generated torsion-free module over a B\'{e}zout domain, hence $W_i/W_{i-1}$ is free by \cite[Corollary V.2.8]{fl01}. So the following exact sequence of $R$-modules splits
\[ 0\to W_{i-1} \to W_i\to W_i/W_{i-1} \to 0, \]
and there exists a free rank $1$ $R$-module $M_i$ such that $W_i=W_{i-1}\oplus M_i$. Hence we obtain a decomposition $R^n = \oplus_{i=1}^n M_i$  such that for all $i=1,\dots,n$, $AM_i\subset \oplus_{j\le i}M_j$.
Now $M_i$ is a free rank $1$ $R$-module, and so the above decomposition gives an $R$-basis of $R^n$ such that with respect to this new basis, $A$ is an upper triangular matrix.
\end{proof}

\begin{lemma}\label{lem:triangularcommutator}
  Let $R$ be a ring and $A\in\Mat_n(R)$ be an upper triangular trace $0$ matrix. Then $A$ is a commutator.
\end{lemma}
\begin{proof}
  Let $X=(x_{ij})\in\Mat_n(R)$ be the matrix with $1$'s on the superdiagonal and $0$ everywhere else, that is,
  \[X=
    \begin{pmatrix}
      0 & 1 & & 0 & 0\\
      0 & 0 & \ddots & 0 & 0\\
      \vdots&\vdots & \ddots  &\ddots & \\
      0 & 0 & \cdots& 0 & 1\\
      0 & 0 & \cdots& 0 & 0
    \end{pmatrix}
    .\]
If $A=(a_{ij})$, then we define $B=(b_{ij})\in\Mat_n(R)$ sequentially from the top to the bottom as follows:
  \[ b_{ij}=
    \begin{cases}
      0 & \text{if }i=1,\\
      a_{1,j}&\text{if }i=2,\\
      a_{i-1,j}+b_{i-1,j-1}&\text{if }3\le i\le n.
    \end{cases}
\]
To ease the notation, we also let $b_{ij}:=0$ if $i=n+1$ or $j=0$. Now for any $1\le i,j\le n$, we have
\[ [X,B]_{ij} = b_{i+1,j}-b_{i,j-1}.\]
Now for $i=1$,
\[ [X,B]_{1,j} = b_{2,j}-b_{1,j-1} = a_{1,j},\]
and so $[X,B]_{1,j}=A_{1,j}$. For $i\ge 2$, we have
\[ [X,B]_{i,j} = b_{i+1,j}-b_{i,j-1} = a_{i,j}+b_{i,j-1}-b_{i,j-1}=a_{i,j},\]
so $[X,B]_{i,j}=A_{i,j}$. Hence $[X,B]=A$.
\end{proof}

We now combine \Cref{thm:bezouttriangularisation} and \Cref{lem:triangularcommutator} to show one of the main theorems of this section.
\bezout
\begin{proof}
  By  \Cref{thm:bezouttriangularisation} we can triangularise $A$. Note that for any $A\in\Mat_n(R)$ and $g\in \GL_n(R)$, $A$ is a commutator if and only if $gAg^{-1}$ is a commutator since for any $B,C\in\Mat_n(R)$,
\[ A=[B,C] \iff gAg^{-1} = [gBg^{-1},gCg^{-1}].\]  
Then \Cref{lem:triangularcommutator} implies that $A$ is a commutator. 
\end{proof}

\begin{remark}
  The only class of rings for which it was previously known that every trace $0$ matrix is a commutator is the class of principal ideal rings, and this is due to Stasinski in \cite[Theorem 6.3]{Stasinski16}. So \Cref{thm:bezoutcommutator} gives a new class of rings where every trace $0$ matrix is a commutator. Note that \Cref{thm:bezoutcommutator} does not imply Stasinski's result since we assume that the B\'{e}zout domain has algebraically closed quotient field. We also note that in our proof, we take a different approach to Stasinski's proof. 
\end{remark}

We now discuss examples of B\'{e}zout domains with algebraically closed quotient field.
\begin{theorem}[{\cite[Theorem 102]{Kaplansky74}}]\label{thm:kapbezout}
  Let $R$ be a Dedekind domain and $K$ be its quotient field. If for all finite extensions $L/K$, the integral closure $S$ of $R$ in $L$ has torsion Picard group $\Pic(S)$, then the integral closure $\widebar{R}$ of $R$ in an algebraic closure of $K$ is a B\'{e}zout domain.
\end{theorem}

\begin{example}
  Let $R:=\Z$ or $R:=k[t]$, with $k$ a finite field, and $K$ be the quotient field of $R$. Then for any finite extension $L/K$, the integral closure $S$ of $R$ in $L$ has finite Picard group. Hence by \Cref{thm:kapbezout}, the integral closure $\widebar{R}$ of $R$ in an algebraic closure of $K$ is a B\'{e}zout domain. Moreover, the quotient field of $\widebar{R}$ is algebraically closed, so by \Cref{thm:bezoutcommutator}, every trace $0$ matrix is a commutator over $\widebar{R}$.
\end{example}

\begin{corollary}\label{cor:dedekindextension}
Let $R$ be a Dedekind domain and $K$ be its quotient field. If for all finite extensions $L/K$, the integral closure $S$ of $R$ in $L$ has torsion Picard group $\Pic(S)$, then for all trace $0$ matrices $A\in\Mat_n(R)$, there exists an $R$-algebra $T$ such that $T$ is finitely generated as an $R$-module and $A$ is commutator in $\Mat_n(T)$.
\end{corollary}
\begin{proof}
  By \Cref{thm:kapbezout}, the integral closure $\widebar{R}$ of $R$ in an algebraic closure of $K$ is a B\'{e}zout domain. Hence by \Cref{thm:bezoutcommutator}, $A$ is a commutator over $\widebar{R}$, so $A=[B,C]$ for some $B,C\in \Mat_n(\widebar{R})$. If we take $T:=R[ b_{ij}, c_{ij} : 1\le i, j\le n]\subset\widebar{R}$, then $B,C\in \Mat_n(T)$, so $A$ is a commutator in $\Mat_n(T)$.
\end{proof}
\begin{remark}
  In \Cref{cor:dedekindextension}, we could also take the integral closure $T'$ of $R$ in the quotient field $F$ of $T$. Then $T'$ will be a Dedekind domain, but not necessary a finitely generated $R$-module, unless $F/K$ is separable or $R$ is finitely generated over a field.
\end{remark}

\begin{remark}
  A related notion to the assumption in \Cref{thm:kapbezout} is pictorsion rings defined in \cite[Definition 0.3]{GLL15}. A commutative ring $R$ is called \emph{pictorsion} if $\Pic(S)$ is torsion for any finite ring map $R\to S$. As noted in \cite{GLL15} before Lemma 8.10, a pictorsion Dedekind domain satisfies the assumption of \Cref{thm:kapbezout} (the assumption is called Condition (T)(a) in \cite[Definition 0.2]{GLL15}). While it is not known whether every Dedekind domain $R$ satisfying the assumption of \Cref{thm:kapbezout} is pictorsion, if we assume that the residue fields at all maximal ideals of $R$ are algebraic extensions of finite fields, then $R$ is pictorsion \cite[Lemma 8.10 (b)]{GLL15}.
\end{remark}

We now discuss the case of hollow endomorphisms of $R$-modules for an arbitrary ring $R$, and then specialise it to nilpotent matrices over Pr\"{u}fer domain in \Cref{thm:dedekindalgebra}.
Recall that an $n\times n$ matrix is \emph{hollow} if all of its diagonal entries are zero (see e.g. \cite[Section 3.1]{Gentle} for a reference on hollow matrices). In the following definition, we generalise the concept of a hollow matrix to an endomorphism of an arbitrary $R$-module with a decomposition.
\begin{definition}
  Let $R$ be a ring and $M$ be an $R$-module. Suppose we have a decomposition $M=\oplus_{k=1}^nM_k$ where $M_k\subset M$ are $R$-submodules. Then for any $A\in\End_R(M)$ and $1\le i,j\le n$, define $a_{ij}:= \pi_iA\iota_j \in \Hom_R(M_j,M_i)$, where $\pi_i\in\Hom_R(\oplus_{k=1}^nM_k,M_i)$ is the projection and $\iota_j\in\Hom_R( M_j,\oplus_{k=1}^nM_k)$ is the inclusion. We can think of $(a_{ij})$ as an $n\times n$ matrix where the $(i,j)$-th entry $a_{ij}$ lies in $\Hom_R(M_j,M_i)$ instead of $R$. We say that $A\in\End_R(M)$ is \emph{hollow with respect to a decomposition} $M = \oplus_{k=1}^nM_k$ if $a_{kk}=0$ for all $k=1,\dots,n$.
\end{definition}

Note that a matrix $A\in\Mat_n(R)$ being hollow is the same as $A\in\End_R(R^n)$ being a hollow endomorphism with respect to the decomposition coming from the standard basis.

Let $R$ be a commutative ring with $C:=\set{r_1,\dots,r_n}\subset R^\times$ such that $r_i-r_j\in R^\times$ for all $i\ne j$. Such a set $C$ is called a \emph{clique of exceptional units}. This term is usually used in the context of number rings, and the cliques were used by Lenstra to construct Euclidean fields in \cite{lenstra76}.

\begin{example}\label{ex:algebraclique}
  Let $k$ be a field and $R$ be a $k$-algebra. Then any $C\subset k^\times \subset R^\times$ is a clique of exceptional units since $x-y\in k^\times\subset R^\times$ for all distinct $x,y\in C$.
\end{example}

\begin{theorem}\label{thm:hollow}
  Let $R$ be a commutative ring with a clique of exceptional units $\set{r_1,\dots,r_n}$ of size $n\ge 1$.
  Let $M$ be a finitely generated $R$-module with a decomposition $M=\oplus_{k=0}^nM_k$, and suppose that $A\in\End_R(M)$ is hollow with respect to that decomposition. Then
  \[A = [X,A'] \]
  for some $X,A'\in\End_R(M)$. In particular, every hollow matrix in $\Mat_{n+1}(R)$.
\end{theorem}
\begin{proof}
  Let $X$ be the diagonal matrix with $r_0:=0,r_1,\dots,r_n$ on the diagonal. We slightly abuse notation here and we denote by $r_i$ the element in $R$ and the endomorphism given by multiplication by $r_i$ on $M_i$ so that we can view $X$ as an element of $\End_R(M)$. If $A=(a_{ij})$, then define $A':=(a'_{ij})\in \End(\oplus_{k=0}^nM_k)$ as follows:
\[a'_{ij} :=
  \begin{cases}
    0 &\text{if }i=j\\
    (r_i-r_j)^{-1}a_{ij} &\text{otherwise}
  \end{cases}.
\]
 Note that $r_1,\dots,r_n$ are units, so $r_0-r_i=-r_i$ is a unit for all $1\le i\le n$. Then
\[[X,A']_{ij} = x_{ii}a'_{ij}-a'_{ij}x_{jj} = r_{i}a'_{ij}-a'_{ij}r_j=(r_i-r_j)a'_{ij}=a_{ij},\]
so $A=[X,A']$. 

The last part follows since a hollow matrix $A\in\Mat_{n+1}(R)$ is hollow with respect to the decomposition of $R^{n+1}$ coming from the standard basis.
\end{proof}

\begin{corollary}
  Let $k$ be a field and $R$ be a $k$-algebra. Then every hollow matrix in $\Mat_n(R)$ is a commutator for any $1\le n\le \# k$. In particular, if $k$ is a field of infinite cardinality, then every $n\times n$ hollow matrix is a commutator for any $n \ge 1$.
\end{corollary}
\begin{proof}
  From \Cref{ex:algebraclique}, we have a clique of exceptional units $C\subset k^\times\subset R^\times$ of any size up to $\# k^\times$. Since every hollow matrix in $\Mat_{\# C+1}(R)$ is a commutator by \Cref{thm:hollow}, every $n\times n$ hollow matrix is a commutator for any $1\le n\le \# k^\times+1=\# k$.  
\end{proof}

We now consider nilpotent matrices over a Pr\"{u}fer domain $R$. We generalise the procedure described by Appleby in \cite[Section 3]{Appleby98} from Dedekind domains to Pr\"{u}fer domains. This allows us to find a decomposition of a finitely generated torsion-free $R$-module $M$ that makes $A$ strictly upper triangular for any fixed nilpotent endomorphism $A\in\End_R(M)$.
\begin{lemma}\label{lem:hollowdecomp}
  Let $R$ be a Pr\"{u}fer domain, $M$ be a finitely generated torsion-free $R$-module and $A\in\End_R(M)$ be a nilpotent endomorphism. Then there exists a decomposition $M=\oplus_{i=1}^nM_i$ where $n=\rk M$, such that the endomorphism $A$ is strictly upper triangular with respect to this decomposition.
\end{lemma}
\begin{proof}
  Let $K$ be the quotient field of $R$ and $V:=M\tensor_R K$. Then $M\subset V$ since $M$ is torsion-free, and we can consider $A\in\End_K(V)$. Since $A$ is nilpotent, we can find a filtration of $K$-vector spaces
  \[  0=V_0\subset V_1\subset \dots \subset V_n=V,\]
  such that $AV_i\subset V_{i-1}$ for all $i=1,\dots, n$, and $\dim_KV_i=i$. If we let $W_i := R^n\cap V_i$ for $i=1,\dots, n$, then we obtain the following filtration of $R$-modules
  \[ 0=W_0\subset W_1\subset \dots \subset W_n = M, \]
  such that $AW_i\subset W_{i-1}$ for all $i=1,\dots,n$. Now, $W_i$ is a finitely generated $R$-module by \Cref{lem:finitemodules}, and so $W_i/W_{i-1}$ is also finitely generated. Now by \Cref{lem:torsionfree}, $W_i/W_{i-1}$ is a torsion-free as well, and so it is a finitely generated torsion-free module over a Pr\"{u}fer domain, hence projective by \cite[Theorem V.2.7]{fl01}. So the following exact sequence splits,
  \[ 0\to W_{i-1}\to W_{i}\to W_{i}/W_{i-1}\to 0, \]
  and we have $W_i = W_{i-1} \oplus M_i$ for some $R$-module $M_i$. So we obtain a decomposition $M = \oplus_{i=1}^n M_i$. Now consider the matrix representation $A=(a_{ij})$ with respect to this decomposition. Then $a_{ij}=0$ for all $1\le j\le i\le n$ since $AM_j\subset W_{j-1}$ and $W_{j-1}\cap M_i=0$. Hence $A$ is strictly upper triangular with respect to the decomposition $M=\oplus_{i=1}^nM_i$.
\end{proof}

\begin{remark}
  For a commutative Noetherian ring $R$, Yohe showed that every nilpotent matrix being similar to a strictly upper triangular matrix is equivalent to $R$ being a principal ideal ring \cite[Theorem 1]{Yohe67}. So for a non-PID Dedekind domain $R$, there are nilpotent matrices that are not similar to a strictly upper triangular matrix. Nevertheless, in \Cref{lem:hollowdecomp}, we show that by allowing the matrix entries to lie in $\Hom_R(M_j,M_i)$ instead of $R$, we can put any nilpotent matrix over $R$ in a strictly upper triangular form.
\end{remark}

\begin{corollary}\label{cor:nilpotentdedekind}
  Let $R$ be a Pr\"{u}fer domain with a clique of exceptional units $C:=\set{r_1,\dots,r_n}$ of size $n\ge 1$, and $M$ be a finitely generated projective $R$-module $M$ of rank $m\le n+1$. Then every nilpotent endomorphism $A\in\End_R(M)$ is a commutator. In particular, every nilpotent matrix in $\Mat_m(R)$ is a commutator for all $1\le m\le n+1$.
\end{corollary}
\begin{proof}
  By \Cref{lem:hollowdecomp}, we can construct a decomposition $M=\oplus_{k=1}^m M_k$ such that $A$ is hollow with respect to this decomposition. Since any subset of a clique of exceptional units is still a clique, we can take a subset of $C$ of size $m-1$ and apply \Cref{thm:hollow} to obtain $X,A'\in\End_R(M)$ such that $A=[X,A']$. 
\end{proof}

\dedekind
\begin{proof}
From \Cref{ex:algebraclique}, we have a clique of exceptional units $k^\times\subset R^\times$. So by  \Cref{cor:nilpotentdedekind}, every nilpotent endomorphism in $\End_R(M)$ is a commutator if $\rk M \le \# k^\times + 1 = \# k$. 
\end{proof}

\Cref{thm:dedekindalgebra} provides partial progress towards answering the following open question.
\begin{question}[{\cite[Section 7]{Stasinski16}}]\label{que:dedekindcommutator}
  Let $R$ be a Dedekind domain and $n\ge 3$. Is every $n\times n$ trace $0$ matrix in $\Mat_n(R)$ a commutator?
\end{question}
Note that a Dedekind domain is a Pr\"{u}fer domain. Lissner answered \Cref{que:dedekindcommutator} affirmatively in \cite[Appendix]{Lissner65} for $n=2$, but no progress has been made since for $n\ge 3$.

\begin{example}
For a fixed $n$, we can provide examples of non-PID number rings $R$ where every nilpotent matrix in $\Mat_m(R)$ is commutator for all $1\le m\le n$. We construct a large clique of exceptional units in $\Z[\zeta_p]$ for a prime $p$ where $\zeta_p= e^{2\pi i/p}$ based on an example in \cite[Section 3]{lenstra76}. Let $\omega_i:=\frac{\zeta_p^i-1}{\zeta_p-1}\in\Z[\zeta_p]$ for any $i=1,\dots, p-1$. Then the absolute norm of $\omega_i$ is 1, so $\omega_i\in \Z[\zeta_p]^\times$ for all $i=1,\dots, p-1$. Now given $1\le i<j\le p-1$, 
\[\omega_j-\omega_i = \zeta_p^i\omega_{j-i}\in\Z[\zeta_p]^\times, \]
so $\set{\omega_1,\dots,\omega_{p-1}}$ is a clique of exceptional units. Hence by \Cref{cor:nilpotentdedekind}, every nilpotent matrix in $\Mat_m(\Z[\zeta_p])$ is a commutator for all $1\le m\le p$. Moreover, $\Z[\zeta_p]$ has class number greater  than $1$ for all primes $p\ge 23$ (see \cite[Theorem 11.1]{washington97}), so $\Z[\zeta_p]$ will not be a PID for those primes. Hence for any prime $p\ge 23$, $\Z[\zeta_p]$ gives a non-trivial example which does not follow from \cite[Theorem 6.3]{Stasinski16}.
\end{example}

\section{Construction of Trace Zero Non-commutators}\label{sec:two}
In this section, we construct a trace $0$ non-commutator assuming we can solve a certain combinatorial problem which is stated in \Cref{q:mainWithd} in the next section. We will first give a couple of definitions that will be used in the problem and the main theorem \Cref{thm:counterexample}.

\begin{definition}
  Let $m\ge 1$ and $d\ge 0$ be integers. We will call a set $S\subset \Z_{\ge 0}^m$ \emph{$d$-separated} if for all distinct $s=(s_i), s'=(s'_i)\in S$,
  \[ |s-s'|_1:=\sum_{i=1}^{m}|s_i-s_i'| > d. \]
\end{definition}

\begin{definition}
  Let $n,r\in\Z_{\ge 0}$. Then the \emph{discrete $n$-simplex of length $r$} is
  \[ \Delta(n,r):=\set{v=(v_i)\in\Z_{\ge 0}^{n+1}}{\sum_{i=1}^{n+1} v_i = r}.\]
\end{definition}

We now give the construction of the trace $0$ non-commutator. Let $R$ be a commutative ring. We use the following notation in the theorem: given a point $s=(s_i)\in\Z^m_{\ge 0}$ and elements $x_1,\dots,x_m\in R$, we let $x^s:=\prod_{i=1}^mx_i^{s_i}\in R$. As usual, we set $r^0:=1$ for any $r\in R$.

\begin{theorem}
  \label{thm:counterexample}
  Let $R$ be a commutative ring with an ideal $I\subset R$ with the following property: there exist non-zero $x_1,\dots,x_m\in I$ with $m\ge 3$, and $d\ge 0$ such that for all $0\le k\le 3d+1$,
  \[I^k/I^{k+1} = \bigoplus_{\substack{t\in\Z_{\ge 0}^m \\ |t|_1=k}}(R/I)(x^t + I^{k+1}).\]
  In other words, $I^k/I^{k+1}$ is a free $R/I$-module and the degree $k$ monomials in the $x_i$'s form an $R/I$-basis of $I^k/I^{k+1}$ for all $0\le k\le 3d+1$. 

Suppose further that there exists a $2d$-separated set $S:=\set{s_1,\dots,s_{2n-1}}\subset{\Delta(m-1,2d+1)}\subset \Z_{\ge 0}^m$ with $n\ge 2$.
 Then
  \[ X := 
    \begin{pmatrix}
      x^{s_1} & x^{s_2} & \cdots& x^{s_{n-1}} & x^{s_n}\\
      x^{s_{n+1}} & 0 & \cdots& 0 & 0\\
      \vdots & \vdots & \ddots & \vdots & \vdots\\
      x^{s_{2n-2}} & 0 & \cdots &0& 0\\
      x^{s_{2n-1}} & 0 &\cdots&0 & -x^{s_1}
    \end{pmatrix}
  \]
  is a trace $0$ non-commutator in $\Mat_n(R)$.
\end{theorem}
\begin{proof}
  Since $\tr(X)= x^{s_1}-x^{s_1}=0$, the only remaining thing we need to show is that $X$ is not a commutator. 
  Suppose by contradiction that $X=[B,C]$ for some $B=(b_{ij}),C=(c_{ij})\in M_n(R)$. Then
  \[[B-b_{nn}I_n,C-c_{nn}I_n] = [B,C] = X,\]
  so we may assume $b_{nn}=0=c_{nn}$ by taking $B-b_{nn}I_n$ and $C-c_{nn}I_n$.

  Now suppose $b_{i1},b_{1i},c_{i1},c_{1i}\in I^{d+1}$ for all $i=1,\dots, n$. Then from $X=[B,C]$, we have
  \[a_{11} = \sum_{i=1}^n (b_{1i}c_{i1} - c_{1i}b_{i1})\in I^{2d+2}.\]
  But $a_{11}=x^{s_1}\in I^{2d+1}$ is part of the $R/I$-basis of $I^{2d+1}/I^{2d+2}$, so $a_{11}\notin I^{2d+2}$, which is a contradiction. Hence to show that $X$ is a non-commutator, we only need to show that $b_{i1},b_{1i},c_{i1},c_{1i}\in I^{d+1}$ for all $i$. 

We will first show that $b_{i1},b_{1i}\in I^{d+1}$ for all $i$. So let $k$ be the minimum integer such that $b_{i1}\in I^k\setminus I^{k+1}$ or $b_{1i}\in I^k\setminus I^{k+1}$ for some $i$. If $k\ge d+1$ or no such $k$ exists then we are done so assume $k\le d$. Without loss of generality, assume that $b_{1j}\in I^k\setminus I^{k+1}$. Then writing $b_{1j}$ in terms of the monomial basis in $I^k/I^{k+1}$, we have 
\[b_{1j} \equiv \sum_{t\in\Delta(m-1,k)}r_tx^{t} \mod I^{k+1},\]
with $r_t\in R$ for all $t\in\Delta(m-1,k)\subset\Z^{m}_{\ge 0}$ such that $r_{t_0}\notin I$ for some $t_0\in \Delta(m-1,k)$.
Now from
  $X=[B,C]$, we have
  \[\tr(BX) = \tr(B[B,C]) = \tr(BBC - BCB) = \tr([B,BC]) = 0. \]
  Since $b_{nn}=0$, we have
  \begin{equation}
  0 =\tr(BX) = \sum_{i=0}^n\sum_{j=0}^n b_{ij}a_{ji} = a_{11}b_{11} + \sum_{i=2}^n b_{i1}a_{1i} + \sum_{j=2}^n  b_{1j}a_{j1}\label{eq:1}.
  \end{equation}
  Now, $a_{1j}=x^{s_j}$ with $|s_j|_1=2d+1$, so
  \[a_{1j}b_{j1} = x^{s_j}b_{j1} \equiv \sum_{t\in\Delta(m-1,k)}r_tx^{s_j+t} \in I^{k+2d+1}/ I^{k+2d+2}.\]
Now $r_{t_0}\in R\setminus I$, and so by considering \cref{eq:1} mod $I^{k+2d+2}$, we see that a term with a monomial $x^{s_j+t_0}$ must also appear in $b_{1i}a_{i1}$ or $b_{i1}a_{1i}$ mod $I^{k+2d+2}$ for some $i$. So suppose $b_{i1}a_{1i}$ contains such a term. Then we have $x^{s_j+t_0} \equiv x^ux^{s_i} \mod I^{k+2d+2}$ where $x^u$ comes from $b_{i1}$ and $|u|_1=k$. So we have $s_j+t_0 = u+s_i$ which implies
  \[|s_i - s_j| = |t_0-u| \le |t_0| + |u| = 2k\le 2d.\]
  This is a contradiction since $S$ is $2d$-separated. Hence $k\ge d+1$, and so $b_{1i},b_{i1}\in I^{d+1}$ for all $i$. By similar argument, $c_{1i},c_{i1}\in I^{d+1}$ as well. Hence no such matrices $B$ and $C$ can exist and $X$ is a non-commutator.
\end{proof}

The following corollary generalises the result of Mesyan \cite[Prop. 20]{Mesyan06}.
\begin{corollary}\label{cor:d0}
  Let $R$ be a ring and suppose that $I:=(x_1,\dots,x_m)$ is an ideal such that $I/I^2$ is a free $R/I$-module with the classes of $x_1,\dots,x_m$ in $I/I^2$ forming an $R/I$-basis. Then for any $2\le n\le \tfrac{m+1}{2}$,
  \[X:=\begin{pmatrix}
      x_1 & x_2 & \cdots & x_{n-1} & x_n\\
      x_{n+1} & 0 & \cdots & 0 & 0\\
      \vdots & \vdots & \ddots & \vdots & \vdots\\
      x_{2n-2} & 0 & \cdots & 0 & 0\\
      x_{2n-1} & 0 & \cdots & 0 & -x_1
  \end{pmatrix}\]
is not a commutator in $\Mat_n(R)$.
In particular, if a Noetherian ring $R$ has Krull dimension $m\ge 3$, then there exists a trace $0$ non-commutator in $\Mat_n(R)$ for any $2\le n\le \tfrac{m+1}{2}$.
\end{corollary}
\begin{proof}
  We will use \Cref{thm:counterexample} to construct the trace $0$ non-commutator. If we take $d=0$, then the basis condition in the theorem is equivalent to $I/I^2$ being a free $R/I$-module. For the set $S$ in the theorem, if $d=0$ then any set $S\subset\Delta(m-1,1)$ is automatically $2d$-separated. Hence we can take $S$ to be $\Delta(m-1,1)=\set{e_1,\dots,e_m}\subset\Z_{\ge0}^m$ where $e_i$'s are the standard basis of $\Z^m$. Now if $n\le \tfrac{m+1}{2}$, then we may take $2n-1$ points inside $\Delta(m-1,1)$. Hence by \Cref{thm:counterexample}, $X$ is a non-commutator. The second part follows since if we take $I\subset R$ to be a maximal ideal with maximum height, then $\dim_{R/I}I/I^2\ge \dim R=m$ (see Theorem 13.5 in \cite{Matsumura}).
\end{proof}

\begin{remark}
  Let $R$ be a ring and $\fm\subset R$ be a maximal ideal such that $\dim_{R/\fm}\fm/\fm^2\ge 3$. Then by \Cref{cor:d0}, there exists a trace $0$ non-commutator in $\Mat_2(R)$. This also follows from \cite[Prop. 20]{Mesyan06}. Such a maximal ideal exists if $R$ is a Noetherian integrally closed domain of dimension $2$ that is not regular. Indeed, by Serre's criterion \cite[Theorem 23.8]{Matsumura}, $R$ being integrally closed implies $R_\fp$ is regular for any prime ideal $\fp\subset R$ with $\height(\fp)\le 1$. But $R$ is not regular, so $R_\fm$ must be singular at some maximal ideal $\fm\subset R$ of $\height(\fm)=2$. At such maximal ideal $\fm$, $\dim_{R/\fm}\fm/\fm^2 > \dim R_\fm=\height(\fm) = 2$ and so there exists a trace $0$ non-commutator in $\Mat_2(R)$.
\end{remark}

The following result of Lissner is a special case of \Cref{cor:d0}.
\begin{corollary}[{\cite[Theorem 5.4]{Lissner61}}]\label{thm:lissner}
  Let $K$ be a field, $n\ge 2$ be an integer and $R=K[x_1,\dots, x_m]$, with $m\ge 2n-1$. Then there exists a trace $0$ non-commutator in $\Mat_n(R)$.
\end{corollary}

We now demonstrate that the condition in \Cref{thm:counterexample} about the structure of $I^k/I^{k+1}$ is relatively easy to satisfy. The following lemma gives a large class of a ring and an ideal satisfying the condition. Recall that given a commutative ring $R$ and an ideal $I\subset R$, we can construct the \emph{associated graded ring} $\gr_I(R) := \oplus_{i\ge 0}I^{i}/I^{i+1}$. Then given an element $r\in R$, we can associate an element $\initial{r}\in\gr_I(R)$ called the \emph{initial form} defined as follows. Let $j\ge 0$ be such that $r\in I^j\setminus I^{j+1}$. If no such $j$ exists, then $r\in\cap_iI^i$, and we take $\initial{r} := 0\in \gr_I(R)$. Otherwise take $\initial{r}:=r+I^{j+1}\in I^j/I^{j+1}\subset\gr_I(R)$.

\begin{lemma}
  \label{lem:regular}
  Let $R$ be a ring and $\fm\subset R$ be a maximal ideal such that $R_\fm$ is a regular local ring (see \cite[Section 14]{Matsumura} for the definition of a regular local ring).
  Then 
  \[\gr_\fm(R) \iso k_\fm[x_1,\dots,x_m], \]
  where $k_\fm:=R/\fm$ is the residue field and $m$ is the dimension of $R_\fm$ (which is finite since $R_\fm$ is a Noetherian local ring) and $k_\fm[x_1,\dots,x_m]$ is a polynomial ring in $m$ variables.
  Moreover, there exist $y_1,\dots,y_m\in \fm$ such that $\initial{y_1},\dots,\initial{y_m}$ map to $x_1,\dots,x_m$ under the above isomorphism and the degree $k$ monomials in the $y_i$'s form a $k_\fm$-basis of $\fm^k/\fm^{k+1}$ for all $k\ge 0$.
\end{lemma}
\begin{proof}
  Since $R_\fm$ is a regular local ring, 
  \[\gr_{\fm R_\fm}(R_\fm)\iso k_\fm[x_1,\dots,x_m],\] 
  as graded rings by Theorem 14.4 in \cite{Matsumura}. Now $\fm$ is maximal, so $\fm^i/\fm^{i+1}\iso (\fm R_\fm)^i/(\fm R_\fm)^{i+1}$ for all $i\ge 0$, and hence $\gr_\fm(R)\iso\gr_{\fm R_\fm}(R_\fm)$. So we have
  \[ \gr_\fm(R)\iso \gr_{\fm R_\fm}(R_\fm)  \iso k_\fm[x_1,\dots,x_m].\]
  They are isomorphic as graded rings, so there exist $y_1,\dots,y_m\in R$ such that $\initial{y_i}$ maps to $x_i$. The above isomorphism also implies that the degree $k$ monomials in $y_i$'s form a $k_\fm$-basis of $\fm^k/\fm^{k+1}$ for all $k\ge 0$ since it is true for the polynomial ring on the right hand side.
\end{proof}

We now state the main result combining \Cref{thm:counterexample} and the combinatorial result \Cref{thm:quadraticS}.
\main 
\begin{proof}
  We will use \Cref{thm:counterexample} to construct the trace $0$ non-commutator, so we need to check the basis condition, and construct the set $S$ in the assumption of the theorem.
  By \Cref{lem:regular}, the basis condition in \Cref{thm:counterexample} is satisfied for all $d$, and by \Cref{thm:quadraticS}, if $m$ is odd, then there exists an $S$ with $\# S= \tfrac{(m+1)^2}{4}$. So if $2n-1\le \tfrac{(m+1)^2}{4}$, then the assumption of \Cref{thm:counterexample} is satisfied and there exists a trace $0$ non-commutator in $M_n(R)$. If we rearrange $2n-1\le \tfrac{(m+1)^2}{4}$, we obtain $n \le  \tfrac{m^2+2m+5}{8}$. For $m$ even, there exists an $S$ with $\# S = 2n-1\le \tfrac{m(m+2)}{4}$ from \Cref{thm:quadraticS}, and we may rearrange the inequality to $n \le \tfrac{m^2+2m+4}{8}$. But if $m$ is even and $n$ is an integer, then $n \le \tfrac{m^2+2m+4}{8}$ is equivalent to $n \le \tfrac{m^2+2m+5}{8}$ so we have the stated inequality for $m$ even as well.
\end{proof}

For an ideal $I$ of a ring $R$, let $\nu(I)$  be the minimal number of generators of $I$. In the case of a Noetherian local ring $R$ and its maximal ideal $\fm$, we have $\nu(\fm) = \dim_{R/\fm}(\fm/\fm^2)$ by Nakayama's Lemma \cite[Theorem 2.3]{Matsumura}.
\begin{corollary}\label{cor:numax}
  Let $R$ be a Noetherian ring and let $n\ge 2$ be such that every trace $0$ matrix in $\Mat_n(R)$ is a commutator. Then for all maximal ideals $\fm$ of $R$,
  \[\nu(\fm) <
    \begin{cases}
      2\sqrt{2n-1} & \text{if }R_\fm\text{ is regular},\\
      2n-1 & \text{otherwise}.
    \end{cases}
  \]
\end{corollary}
\begin{proof}
  Suppose $R_\fm$ is regular. Since every $n\times n$ trace $0$ matrix is a commutator, $n >\frac{m^2+2m+5}{8}$ where $m=\dim R_\fm = \dim_{R/\fm}(\fm/\fm^2) = \nu(\fm R_\fm)$ by \Cref{thm:main}. Hence $2\sqrt{2n-1}-1> m=\nu(\fm R_\fm)$.
  Now by \cite[Theorem 1]{dg77}, $\nu(\fm R_\fm) +1 \ge  \nu(\fm)$ if $R_\fm$ is regular, so 
\[ 2\sqrt{2n-1}> \nu(\fm).\]  
Hence the stated inequality follows. 

  For $R_\fm$ not regular, we apply \Cref{cor:d0} to obtain $n > \frac{m+1}{2}$ where $m=\dim_{R/\fm}(\fm/\fm^2) = \nu(\fm R_\fm)$. Hence
  \[ 2n-1 > m=\nu(\fm R_\fm).\]
  Now by \cite[Theorem 1]{dg77}, $\nu(\fm R_\fm ) = \nu(\fm)$ if $R_\fm$ is not regular, so the stated inequality follows.
\end{proof}

\begin{remark}
If $R$ is a Noetherian ring such that every $2\times 2$ trace $0$ is a commutator, then \Cref{cor:numax} implies 
\[\nu(\fm) \le
  \begin{cases}
    3 &\text{if }R_\fm\text{ is regular},\\
    2 &\text{otherwise},
  \end{cases}
\]
for any maximal ideal $\fm\subset R$.
However, it is known that $\nu(\fm)\le 2$ for any maximal ideal $\fm$ if every $2\times 2$ trace $0$ matrix is a commutator (see \cite[(3) Lemma, (4) Lemma]{BDG80}). 
Note that we do not know whether every $2\times 2$ trace $0$ matrix is a commutator if $\nu(\fm)\le 2$ for all maximal ideals $\fm$. See however, \Cref{thm:local} and \Cref{cor:generatedby2} for cases where this is true.
\end{remark}

In view of \Cref{thm:main}, it is natural to ask the following question.
\begin{question}
  Let $R$ be a Noetherian ring such that there exists $c>0$ with $\dim_{R/\fm}\fm/\fm^2\le c$ for all maximal ideals $\fm\subset R$. Does there exist an $n_0=n_0(R)\ge 1$ such that for all $n\ge n_0$, every trace $0$ matrix in $\Mat_n(R)$ is a commutator?
\end{question}
The example produced in \Cref{thm:unbounded} shows that the answer to this question is negative if the assumption on the existence of such $c$ is omitted. The simplest rings for which the question is open are regular local rings of dimension $2$ and Dedekind domains. For both of those types of rings, it is not known whether every $n\times n$ trace $0$ matrix is a commutator for any $n\ge 3$.

Let us now summarise below the largest trace $0$ non-commutator we can construct for a fixed ring and an ideal using \Cref{thm:counterexample}.
\begin{theorem}\label{thm:largestn}
  Let $R$ be a polynomial ring in $m\ge 3$ variables over a field and $d \ge 0$ or, more generally, let $R$ be a ring with an ideal $I\subset R$ satisfying the basis condition from \Cref{thm:counterexample}. That is, there exist some non-zero $y_1,\dots,y_m\in I$ with $m\ge 3$ such that the degree $k$ monomials in the $y_i$'s form an $R/I$-basis of $I^k/I^{k+1}$ for all $0\le k\le 3d+1$.  Then the following list describes the sizes $n$ for which we can construct a trace $0$ non-commutator in $\Mat_n(R)$.
  \begin{enumerate}
  \item If $d=0$, then any $2\le n\le\tfrac{m+1}{2}$.
  \item If $d=1$ and $m\not\equiv 5\pmod{6}$, then any $2\le n\le\tfrac{1}{2}(\floor{\frac{m}{3}\floor{\frac{m-1}{2}}}+ m + 1)$.
  \item If $d=1$ and $m\equiv 5\pmod{6}$, then any $2\le n\le\tfrac{1}{2}(\floor{\frac{m}{3}\floor{\frac{m-1}{2}}}+ m)$.
  \item If $m=3$, then $n=2$.
  \item If $m\ge 4$ and $d\ge m-1$, then any $2\le n\le \frac{m^2+2m+5}{8}$.
  \item For certain specific $m$ and $d$, the corresponding entry in the table below is the size of the largest non-commutator matrix that \Cref{thm:counterexample} can construct.  When the entry is in bold, it is bigger than the upperbound provided by (5).
\begin{table}[H]
  \centering
    \begin{tabular}{|>{\bfseries}c| *{11}{c}|}
      \hline
      \diagbox[width=2.5em]{m}{\kern-1em d} & \textbf{2} & \textbf{3} & \textbf{4} & \textbf{5} & \textbf{6} & \textbf{7} & \textbf{8} & \textbf{9} & \textbf{10} & \textbf{11} &\textbf{12} \\ \hline
      4& 3 & 3 & \textbf{4} & \textbf{4} & \textbf{4} & \textbf{4} & \textbf{4} 
& \textbf{4} &  \textbf{4} & \textbf{4} & \textbf{4}\\
      5& 5 & 5 & 5 & \textbf{6} & \textbf{6} & \textbf{6} & \textbf{6} & \textbf{6} 
&  \textbf{7} &   & \\
      6& 6 & 7 & \textbf{8} &   &   &   &   &   &    &   &\\
      7& 9 &   &   &   &   &   &   &   &    &   &\\ 
      8& 12 &   &   &   &   &   &   &   &    &   &\\\hline
    \end{tabular}
   \caption{Size of the largest non-commutator matrix that \Cref{thm:counterexample} can construct}
 \end{table}
  \end{enumerate}
\end{theorem}
\begin{proof}
  We will use \Cref{thm:counterexample} to construct the trace $0$ non-commutator. 

  For (1), the proof is in \Cref{cor:d0}. Note that the assumptions of the \Cref{thm:counterexample} can be simplified to the assumptions in \Cref{cor:d0} (see the proof of \Cref{cor:d0} for more details on the assumptions). 

  For (2) to (6) we will first describe the set $S$ used in  \Cref{thm:counterexample}.

  For (2) and (3) we will use \Cref{thm:binary}. Given $d=1$ and $m\ge 0$, take $S$ to be the maximum set of binary vectors of length $m$ that are Hamming distance at least $4$ apart and constant weight 3.   

  For (4), \Cref{prop:smallR} describes how to obtain the set $S$ for $m=3$ and for any $d$.
  For (5), \Cref{thm:quadraticS} describes how to obtain the set $S$ for any $m\ge 1$ and $d\ge m-1$.
  For (6) \Cref{prop:dataSize} describes the largest $S$ one can obtain for the listed $m$ and $d$.

  Once the $S$ needed for \Cref{thm:counterexample} is obtained, we can construct an $n\times n$ trace $0$ non-commutator if $2n-1\le \# S$. Note that for (5), we have $\# S = \tfrac{m(m+2)}{4}$ when $m$ is even. This inequality can be rearranged to $2\le n\le \frac{m^2+2m+4}{8}$, but this is equivalent to $2\le n\le \frac{m^2+2m+5}{8}$ since $m$ is even. Hence we obtain the inequality stated in (5) for both even and odd $m$.
\end{proof}

\begin{remark}
  The $n$'s given in the table in \Cref{thm:largestn} is the largest size for which we can construct a trace $0$ non-commutator using \Cref{thm:counterexample}. On the other hand, (5) in the list can be improved if we can construct a bigger $S$ required for \Cref{thm:counterexample}. For example, if $m=4$, (5) gives $2\le n\le 3$, but we see that in the table in (6), we can take $n=4$ if $d\ge 4$.
\end{remark}

While the upper bound on $n$ in \Cref{thm:main} could be improved, the size of the non-commutator we can construct using \Cref{thm:counterexample} is bounded for a fixed ring $R$ and ideal $I\subset R$, as we now show. 

\begin{prop}\label{prop:upperbound}
  Let $R$ be a commutative ring and let $I\subset R$ be an ideal such that $I/I^2$ is a free $R/I$-module with rank $m\ge 3$. Then the size $n$ of a trace $0$ non-commutator one can construct using \Cref{thm:counterexample} is bounded above by $2^{2m-3}$.
\end{prop}
\begin{proof}
  If $m$ is finite, then the size of the set $S=\set{s_1,\dots,s_{2n-1}}$ required in \Cref{thm:counterexample} is bounded above by $4^{m-1}$ by \Cref{cor:upperbound}. So we  have $2n-1\le 4^{m-1}$ which implies $n\le 2^{2m-3}$. Hence the largest trace $0$ non-commutator one can construct using \Cref{thm:counterexample} is bounded above by $2^{2m-3}$.
\end{proof}

\begin{remark}
\Cref{prop:upperbound} gives an upperbound for the size of the non-commutator one can construct using \Cref{thm:counterexample} for a fixed $R$ and $I$. However, if we only fix $R$, then the size of the non-commutator can be arbitrary large in general. We will later construct a ring with such a property in \Cref{thm:unbounded}.
\end{remark}

\section{Combinatorics}\label{sec:combinatorics}
We will now discuss  the set $S$ required in \Cref{thm:counterexample}. Namely, we would like to answer the following questions.
\begin{question}
  \label{q:mainWithd}
  Given $m\ge 1$ and $d\ge 0$, how large can a set $S\subset \Z^m_{\ge 0}$ be if it is $2d$-separated and contained in $\Delta(m-1,2d+1)$?
\end{question}
\begin{question}
  \label{q:main}
  Given $m\ge 1$, how large can a set $S$ be if it is $2d$-separated and contained in $\Delta(m-1,2d+1)\subset \Z^m_{\ge 0}$ for some $d\ge 0$?
\end{question}
An answer for \Cref{q:mainWithd} will give the size of the largest matrix we can construct using \Cref{thm:counterexample} for a fixed $m$ and $d$. While an answer to \Cref{q:main} will be useful in the situation of \Cref{thm:main} where we are allowed to take arbitrary large $d$.

To answer the above questions, we will first use the following lemma to simplify the problem.
\begin{lemma}
  \label{lem:corner}
  Let $m\ge 1$ and $d\ge 0$ be integers, and suppose that a set $S$ is $2d$-separated and contained in $\Delta(m-1,2d+1)$. Then there exists a set $S'$ such that $S'$ is $2d$-separated, contained in $\Delta(m-1,2d+1)$, $\#S'\ge \#S$ and $v_1:=(2d+1,0\dots,0),\dots,v_m:=(0,\dots,0,2d+1)\in S'$.
\end{lemma}
\begin{proof}
  For each $i$, we will either add $v_i$ to $S$ or replace one of the point in $S$ with $v_i$ to construct $S'$. So let $1\le i\le m$. If $v_i\in S$, then we do not need to modify $S$. Assume now that $v_i\notin S$. If $v_i$ is $2d$-separated from all the points in $S$, then we may take $S' = S\cup \set{v_i}$. Otherwise there is a point $s=(s_j)\in S$ such that $|s-v_i|_1\le 2d$. We will show that such $s$ is unique, and so we may replace $s$ with $v_i$ to obtain $S'$.
So suppose we have $t=(t_j)\in S$ with $|t-v_i|_1\le 2d$, and without loss of generality, assume $s_i\ge t_i$. Now
\[2d \ge |s-v_i|_1 = \sum_{j\ne i} s_j + 2d+1-s_i=4d+2 - 2s_i,\]
so $s_i\ge d+1$, and similarly, $t_i\ge d+1$. Hence
  \begin{align*}
    |s-t|_1 &=  s_i-t_i + \sum_{j\ne i}|s_j-t_j| \\
            &\le s_i-t_i + \sum_{j\ne i}s_j + \sum_{j\ne i}t_j\\
            &\le s_i-t_i + (2d+1-s_i)+(2d+1-t_i) \\
            &= 4d +2-2t_i\\
            &\le 2d.
  \end{align*}
  Since $S$ is $2d$-separated, this implies that $s=t$. Hence there is a unique point in $s\in S$ such that $|s-v_i|_1\le 2d$, and so if we take $S' := (S\setminus\set{s}) \cup \set{v_i}$, then $S'$ will be $2d$-separated.
\end{proof}

We will call the $v_i$'s the \emph{corner points} since they are on the corners of the simplex $\Delta(m-1,2d+1)$. We now consider the necessary and sufficient conditions for a point $s\in\Delta(m-1,2d+1)$ to be $2d$-separated from all the corner points.
\begin{lemma}\label{lem:boundedcoordinate}
  Let $d\ge 0$ and $m\ge 1$ be integers and $s=(s_j)\in\Delta(m-1,2d+1)\subset\Z_{\ge0}^m$. Then $s$ is $2d$-separated from all the corner points if and only if $s_j\le d$ for all $j=1,\dots,m$.
\end{lemma}
\begin{proof}
  First, note that if $s=(s_j)\in\Delta(m-1,2d+1)$, then for all $i=1,\dots,m$,
  \[|s-v_i|_1 = 2d+1-s_i + \sum_{j\ne i}s_j = 2d+1-2s_i+ \sum_{j=1}^ms_j = 2(d-s_i) + 2d+2.\]
  So $s$ is $2d$-separated from $v_i$ if and only if $2(d-s_i)+2d+2 > 2d$, which can be rearranged to
  \[d+1 > s_i. \]
  Since $d$ and $s_i$ are integers, this inequality is equivalent to $d\ge s_i$.
  Hence $s$ is $2d$-separated from all the corner points if and only if $d\ge s_i$ for all $i$.
\end{proof}

If we are constructing a maximum $2d$-separated set $S\subset\Delta(m-1,2d+1)$, then by \Cref{lem:corner}, we can assume that the set $S$ contains all the corner points. Now consider the rest of the points $T:=S\setminus\set{v_1,\dots,v_m}$. For $S$ to be $2d$-separated, $T$ also needs to be $2d$-separated, and in addition, by \Cref{lem:boundedcoordinate}, for all $t=(t_j)\in T$, we must have $t_j\le d$ for all $j$. So in other words, \Cref{q:mainWithd} about the maximum size of $S$ is equivalent to the following question. Also note that the sizes of the largest $S$ in \Cref{q:mainWithd} and the largest $T$ in the following question are related by $\# S=m+\#T$.
\begin{question}\label{q:t}
  Given $m$ and $d$, how large can a set $T$ be if it is $2d$-separated, contained in $\Delta(m-1,2d+1)$, and for all $t=(t_i)\in T$, $t_i\le d$ for all $i=1,\dots,m$.
\end{question}

The $d=1$ case in \Cref{q:t} can be reformulated in terms of binary vectors as follows. Recall that a binary vector of length $m$ is any vector in $\set{0,1}^m$. The weight of such a vector is the sum of all the entries, and for two binary vectors $v,v'$, the Hamming distance between $v$ and $v'$ is $|v-v'|_1$. Now if $d=1$, then every entry in $t\in T$ is either $0$ or $1$, so we can consider $t$ as a binary vector of length $m$. Moreover, since $t\in\Delta(m-1,3)$, the weight of $t$ is 3. In other words \Cref{q:t} for the $d=1$ case is asking for a largest set of binary vectors of length $m$ with weight $3$ that are pairwise Hamming distance being more than $2$ apart. 

The following theorem answers how large such a set of binary vectors can be.  We state the theorem as in \cite{BSSS90}. For the proof, see Theorem 3 in \cite{Schoenheim66}. Note that the Hamming distance between two binary vectors of the same weight must be even, so Hamming distance being more than $2$ apart is equivalent to Hamming distance being at least $4$ apart.
\begin{theorem}[{\cite[Theorem 4]{BSSS90}}]\label{thm:binary}
  Let $A(m,d,w)$ be the maximal possible number of binary vectors of length $m$, Hamming distance at least $d$ apart and constant weight $w$. Then
  \[A(m,4,3) =
    \begin{cases}
      \floor{\frac{m}{3}\floor{\frac{m-1}{2}}}, & \text{if }m\not\equiv 5\pmod{6}\\
      \floor{\frac{m}{3}\floor{\frac{m-1}{2}}} - 1, & \text{if }m\equiv 5\pmod{6}
    \end{cases}
\]
\end{theorem}
So for all $m\ge 1$ and $d=1$, the largest $T$ in \Cref{q:t} has size $A(m,4,3)$, while the largest $S$ in \Cref{q:mainWithd} has size $A(m,4,3)+m$, answering \Cref{q:t} and \Cref{q:mainWithd} respectively.

For small $m$'s we can answer \Cref{q:main} fully.
\begin{prop}\label{prop:smallR}
  Given $m=1,2$ or $3$, the largest set $S$ that is $2d$-separated and contained in $\Delta(m-1,2d+1)\subset \Z_{\ge 0}^m$ has size $\# S=1,2$ and $4$ respectively.
\end{prop}
\begin{proof}
  By \Cref{lem:corner}, we may assume that the corner points are already in $S$. Then for $m=1$ and $m=2$, there are no points in $\Delta(m-1,2d+1)$ that are $2d$-separated from the corner points. So $\# S=m$ is the largest possible set.

  For $m=3$, suppose there exist $2$ distinct points $s,s'\in\Delta(2,2d+1)$ that are $2d$-separated from the corner points. If $s=(a,b,c)$ and $s'=(a',b',c')$, then we have $0\le a,b,c,a',b',c' \le d$ by \Cref{lem:boundedcoordinate}. Since $|s|_1=2d+1=|s'|_1$, we may assume without loss of generality that $a\ge a'$, $b\le b'$ and $c\le c'$. Then
  \begin{align*}
    |s-s'|_1 &= |a-a'| + |b-b'| + |c-c'| \\
             &= a-a' + b'-b + c'-c\\
             &= 2d+1 - 2a' -(2d+1) + 2a'\\
             &= 2(a-a')\\
             &\le 2d
  \end{align*}
  so $s$ and $s'$ cannot be $2d$-separated. Hence $S$ can only contain one more point apart from the corner points, and hence $\# S=4$ is the largest.
\end{proof}

For other small $m$'s, we can use a computer program to determine explicitly how large $S$ can be for small $d$'s. But we will first convert that question into a graph theory problem. Given $m\ge 1$ and $d\ge 0$ integers, let $G(m,d):=(V,E)$ be the graph with a vertex set $V$ and an edge set $E$ defined as follows: The vertices are points in  $\Delta(m-1, 2d+1)$ with entries at most $d$, and an edge exists between distinct $s, s' \in V$ if 
  \[ |s-s'|_1=\sum_{i=1}^{m}|s_i-s_i'| \le 2d.\]
  Recall that for a graph, an \emph{independent set} is a set of vertices with no edges between them, and an independent set with the largest size is called a \emph{maximum independent set}. The cardinality of such a maximum independent set is called the \emph{independence number} of the graph. The set $T$ in \Cref{q:t} is an independent set of $G(m,d)$, so we would like to know the independence number of $G(m,d)$ for a given $m\ge 1$ and $d\ge 0$.
\begin{prop}\label{prop:dataSize}
  For the following $m$ and $d$, the size of the largest $2d$-separated set ${S\subset \Z_{\ge0}^m}$ contained in $\Delta(m-1,2d+1)$ is given in the table below.
\begin{table}[H]
  \centering
\begin{tabular}{|>{\bfseries}c| *{12}{c}|}
      \hline
    \diagbox[width=2.5em]{m}{\kern-1em d}& \textbf{1} & \textbf{2} & \textbf{3} & \textbf{4} & \textbf{5} & \textbf{6} & \textbf{7} & \textbf{8} & \textbf{9} & \textbf{10} & \textbf{11} &\textbf{12} \\ \hline
                    4& 5 & 6 & 6 & 7 & 7 & 7 & 7 & 7 & 8 &  8 & 8 & 8\\
                    5& 7 & 10& 10& 10& 11& 11& 12& 12& 12& 13 &   & \\
                    6& 10& 12& 14& 15&   &   &   &   &   &    &   &\\
                    7& 14& 18&   &   &   &   &   &   &   &    &   &\\ 
                    8& 16& 24&   &   &   &   &   &   &   &    &   &\\\hline
    \end{tabular}
   \caption{Size of the largest $2d$-separated set ${S\subset \Z_{\ge0}^m}$ contained in $\Delta(m-1,2d+1)$}
 \end{table}
\end{prop}
\begin{proof}
  As stated before the proposition, we would like to know the independence number of $G(m,d)$ to answer \Cref{q:t}, and so we used the function \texttt{IndependenceNumber} in Magma\cite{Magma} to compute it. 
As noted in the paragraph before \Cref{q:t}, the size of $T$ in \Cref{q:t} is related to the size of $S$ in \Cref{q:mainWithd} by $\# S = \# T + m$. So we added $m$ to the output from \texttt{IndependenceNumber} to obtain the size of the largest set $S\subset \Delta(m-1,2d+1)$ that is $2d$-separated, thus answering \Cref{q:mainWithd} completely for these values of $m$ and $d$.

The missing entries in the table are due to the computations taking too long for those parameters.
\end{proof}
\begin{remark}
  The independence number of a graph  is a well known invariant and has been well studied. In particular, there are inequalities which involves the independence number and other invariants of a graph. See \cite[Chapter 3]{Willis11} for a list of such inequalities. Unfortunately the inequalities listed were either not strong enough, or had an invariant that was difficult to compute, and we could not obtain any useful bound for our graphs. 
\end{remark}

For any $m\ge 1$, we have the following construction of a set $S$ suitable for use in \Cref{thm:counterexample}.
\begin{theorem}\label{thm:quadraticS}
  Let $m\in\Z_{>0}$. Then for any $d\ge m-1$, there exists a $2d$-separated set $S\subset {\Delta(m-1,2d+1)}\subset \Z_{\ge 0}^m$ with
  \[\#S =
    \begin{cases}
      \frac{m(m+2)}{4} & \text{if $m$ is even},\\
      \frac{(m+1)^2}{4} & \text{if $m$ is odd}.
    \end{cases}
  \]
\end{theorem}
\begin{proof}
  We will explicitly construct such an $S$.
For $1\le j\le \tfrac{m}{2}$, $1\le k\le m-2j$ and $r:=\ord_2(j)$, let
\[v_{j,k}:=(\underbrace{0,\dots,0}_{k-1},d-r-k,\underbrace{0,\dots,0}_{j-1},d,\underbrace{0\dots,0}_{j-1},r+k+1,\underbrace{0,\dots,0}_{m-2j-k})\in\Delta(m-1,2d+1).\]
To check that $v_{j,k}\in\Delta(m-1,2d+1)$, we first verify that $d-r-k=d-\ord_2(j)-k\ge 0$. \begin{align*}
  d-\ord_2(j) - k &\ge d-\ord_2(j) - (m-2j)\\
  &=d-m -\ord_2(j)+2j\\
  &\ge d - m  + 2\\
  &\ge 0
\end{align*}
where the last inequality follows from the assumption $d\ge m-1$. Looking at the length, we have $|v_{j,k}|_1 = d-r-k+d+r+k+1=2d+1$, so it has the correct length. Hence $v_{j,k}$ is a valid point in $\Delta(m-1,2d+1)$.

We take 
\[S:=\set{\text{corner points of }\Delta(m-1,2d+1)}\cup \set{v_{j,k}}{1\le j\le \tfrac{m}{2}, 1\le k\le m-2j}.\]
We first check if all the entries of $v_{j,k}$ are at most $d$ to verify that they are $2d$-separated from the corner points. Since $d-r-k\le d$, the only other value we need to check is $r+k+1$, and
\[r+k+1=\ord_2(j) + k + 1\le \ord_2(j) + m-2j + 1\le m - 2 + 1=m-1\le d.\]
Hence $v_{j,k}$ is $2d$-separated from the corner points.

Now we check that $v_{j,k}$'s are $2d$-separated. So fix two distinct $v_{j,k}$ and $v_{j',k'}$, and let 
\[K:=\set{1\le i\le m}{(v_{j,k})_i\ne 0} \quad\text{and}\quad K':=\set{1\le i\le m}{(v_{j',k'})_i\ne 0}\]
be the support of $v_{j,k}$ and $v_{j',k'}$ respectively. Note that $\# (K\cap K')\le 2$, since $\# K=3=\# K'$ and $K=K'$ if and only if $v_{j,k}=v_{j',k'}$. Then
\begin{align*}
  |v_{j,k}-v_{j',k'}|_1  =&  \sum_{i\in K\cap K'}|(v_{j,k})_i-(v_{j',k'})_i| + 
  \sum_{i\notin K\cap K'}(v_{j,k})_i +   \sum_{i\notin K\cap K'}(v_{j',k'})_i \\
  \ge& \sum_{i\notin K\cap K'}(v_{j,k})_i + \sum_{i\notin K\cap K'}(v_{j',k'})_i.
\end{align*}
Now suppose $\# (K\cap K')\le 1$. Since the entries of $v_{j,k}$ are bounded by $d$, we have $\sum_{i\notin K\cap K'}(v_{j,k})_i\ge 2d+1-d = d+1$, and similarly for $v_{j',k'}$. Hence
\[ |v_{j,k}-v_{j',k'}|_1  \ge \sum_{i\notin K\cap K'}(v_{j,k})_i + \sum_{i\notin K\cap K'}(v_{j',k'})_i\ge (d+1)+(d+1)  > 2d,\]
and so $v_{j,k}$ and $v_{j',k'}$ are $2d$-separated. 

Now we consider the case when $\#(K\cap K') = 2$. There are three possible cases:
\begin{enumerate}
\item $j'=j$ and $k' = k+j$ (where the $k+j$-th position and the $k+2j$-th position agree),
\item $j'=2j$ and $k'=k$ (where the $k$-th position and the $k+2j$-th position agree), or
\item $j'=2j$ and $k'=k-j'$ (where the $k-2j$-th position and the $k$-th position agree).
\end{enumerate}
Note that if $j'\ne j$ and $j'\ne 2j$, then the spacing between the non-zero entries does not match, so the supports can only intersect at most 1 entry.

We first consider the case when $j'=j$ and $k'=k+j$. In this case, we have
\begin{align*}
&|v_{j,k}-v_{j,k'}|_1 \\
=& |(\overbrace{\underbrace{0,\dots,0}_{k-1},d-r-k,\underbrace{0,\dots,0}_{j-1}}^{k'-1},r+k',\underbrace{0\dots,0}_{j-1},r+k+1-d,\underbrace{\overbrace{0,\dots,0}^{j-1}, -r-k'-1,\overbrace{0\dots,0}^{m-2j-k'}}_{m-2j-k})|_1\\
  =&d-r-k+r+k'+d-(r+k+1)+r+k'+1\\
  =&2d+2k'-2k\\
    >&2d
\end{align*}
so they are $2d$-separated. 

If $j'=2j$ and $k'=k$ then we have
  \begin{align*}
    &|v_{j,k}-v_{j',k'}|_1\\
   =&  |(\overbrace{\underbrace{0,\dots,0}_{k-1}}^{k'-1},-r+r',\overbrace{\underbrace{0,\dots,0}_{j-1},d,\underbrace{0\dots,0}_{j-1}}^{j'-1},r+k+1-d,\underbrace{\overbrace{0,\dots,0}^{j'-1}, r'+k'+1,\overbrace{0,\dots,0}^{m-2j'-k'}}_{m-2j-k})|_1\\
    =& -r+r' + d + d-(r+k+1) + r'+k'+1\\
    =&2d + 2(r'-r) \\
    =& 2d + 2(\ord_2(2j) - \ord_2(j))\\
    >& 2d
  \end{align*}
  so these points are $2d$-separated.

  Finally if $j'=2j$ and $k'=k-j'$, then we have
  \begin{align*}
    &|v_{j,k}-v_{j',k'}|_1\\
   =&  |(\underbrace{\overbrace{0,\dots,0}^{k'-1},d-r'-k',\overbrace{0,\dots,0}^{j'-1} }_{k-1},-r-k,\overbrace{\underbrace{0,\dots,0}_{j-1},d,\underbrace{0,\dots,0}_{j-1}}^{j'-1},r+k-r'-k',\underbrace{\overbrace{0,\dots,0}^{m-2j'-k'}}_{m-2j-k})|_1\\
    =& d-r'-k'+r+k+d+r+k-r'-k'\\
    =& 2d + 2(k-k') + 2(r-r')\\
    >& 2d,
  \end{align*}
and those points are also $2d$-separated.

Finally we count how many points we get from this construction.
If $m$ is even, this gives total of
\[m+\sum_{j=1}^{m/2}(m-2j) = \frac{m(m+2)}{4} \]
points, and if $m$ is odd,
\[m+\sum_{j=1}^{(m-1)/2}(m-2j) = \frac{(m+1)^2}{4} \]
points.
\end{proof}
\begin{remark}
  One could add more points to the set $S$ given in \Cref{thm:quadraticS} by considering points with larger support. If a point has a support of size $5$ and has all its non-zero entries taking values around $\tfrac{2d+1}{5}$, then it will be $2d$-separated from any points in $S$. More generally, if $u'$ is the size of the support of the previous point added, then one could add a point with support of size $u>2u'$ with all its non-zero entries taking values around $\tfrac{2d+1}{u}$. While this does increase the size of $S$, it is only useful for large $m$, and it does not appear to increase the asymptotic of the size of $S$.
\end{remark}

We conclude the section with an upper bound on the size of a $2d$-separated set $S$ contained in $\Delta(m-1,2d+1)$.
We will first give an equivalent formulation of \Cref{q:main} by putting all the $d$'s together that will be useful for proving the upper bound.

\begin{question}
  \label{q:real}
  Let $m\in \Z_+$. Consider the regular $(m-1)$-simplex
  \[\Delta_{m-1} := \set{v\in\R^m}{v_i\ge 0, \sum_{i=1}^mv_i = 1}.\]
  What is the maximum $S\subset\Delta_{m-1}$ that is $1$-separated?
\end{question}

\begin{remark}\label{rem:convertQs}
  We can convert between the $S$'s in \Cref{q:main} and \Cref{q:real} in the following way. If a set $S\subset\Delta(m-1,2d+1)$ is $2d$-separated, then $S':=\set{\tfrac{1}{2d+1}s}{s\in S}\subset \Q_{\ge 0}^m$  is contained in $\Delta_{m-1}$. Now if $s,t\in S$ are distinct points, then $|s-t|_1>2d$. But since $|s|_1=|t|_1$ and $s,t\in\Z^m_{\ge 0}$, the distance $|s-t|_1$ must be even, so $|s-t|_1\ge 2d+2$. Hence $S'$ is $1$-separated and satisfies the conditions in \Cref{q:real}.

On the other hand, suppose $S\subset\Delta_{m-1}$ is $1$-separated. For each $s\in S$, take a nearby point $s':=(s'_i)\in\Delta_{m-1}\cap \Q^m$ such that the denominators of the $s'_i$'s are odd. Let $S'$ be the set of all these $s'$'s. Since being $1$-separated is an open condition, if we take $s'$ to be close enough to $s$, then $S'$ will be $1$-separated. Now suppose $2d+1$ is the LCM of the denominators appearing in the entries of all $s'\in S'$. Then $S'':=\set{ (2d+1)s'}{s'\in S'}\subset \Z^m_{\ge 0}$ is contained in $\Delta(m-1,2d+1)$ and is $2d$-separated. Hence $S''$ satisfies the conditions in \Cref{q:main}.
\end{remark}

\begin{prop}\label{prop:realUpperbound}
  Let $m\in\Z_{>0}$. Then for any $S\subset\Delta_{m-1}$ that is $1$-separated, we have
  \[ \# S \le 4^{m-1}.\]
\end{prop}
\begin{proof}
  Suppose we have such an $S$. Then around each point $s\in S$, we have a radius $\tfrac{1}{2}$ (in $\ell_1$-norm) ball $B(s,\tfrac{1}{2})$ such that none of these balls intersect each other. So we have
  \[ \sum_{s\in S} \vol\left(B(s,\tfrac{1}{2})\cap \Delta_{m-1}\right) \le \vol(\Delta_{m-1}).\]
  The volume of $B(s,\tfrac{1}{2})\cap\Delta_{m-1}$ is the smallest when $s$ is a corner point of $\Delta_{m-1}$ since $\Delta_{m-1}$ is a convex polytope. For a corner point $v_1:=(1,0,\dots,0)$,
  \[ B(v_1,\tfrac{1}{2})\cap\Delta_{m-1} =\set{(u_i)\in\R^m}{ \frac{3}{4}\le u_1\le 1, 0\le u_i\le \frac{1}{4}\; \forall\; 2\le i\le m,\sum_{i=1}^mu_i=1},\]
  which is a $(m-1)$-simplex with vertices $(1,0,\dots,0)$, $(\tfrac{3}{4},\tfrac{1}{4},0,\dots,0),\dots, (\tfrac{3}{4},0,\dots,0,\tfrac{1}{4})$. This simplex has side length $\tfrac{1}{4}$ of the side length of the simplex $\Delta_{m-1}$, so 
  \[\#S\left(\tfrac{1}{4}\right)^{m-1}\vol(\Delta_{m-1})\le \sum_{s\in S} \vol\left(B(s,\tfrac{1}{2})\cap \Delta_{m-1}\right) \le \vol(\Delta_{m-1}). \]
  Hence $\# S \le 4^{m-1}$.
\end{proof}
We can translate the upper bound for \Cref{q:real} to an upper bound for \Cref{q:main}.
\begin{corollary}
  \label{cor:upperbound}
  Let $m\ge 1$ be an integer. Suppose $S\subset\Delta(m-1,2d+1)\subset\Z_{\ge 0}^m$ is a $2d$-separated set for some $d\ge 0$. Then $\# S \le 4^{m-1}$.
\end{corollary}
\begin{proof}
  By \Cref{rem:convertQs}, any set $S\subset\Delta(m-1,2d+1)$ that is $2d$-separated can be converted to a set $S'\subset \Delta_{m-1}$ that is $1$-separated. By \Cref{prop:realUpperbound}, $\# S'\le 4^{m-1}$, so $\# S= \# S' \le 4^{m-1}$.
\end{proof}

\section{A Ring with Trace Zero Non-commutators of Arbitrary Large Size}\label{sec:ring}
In this section, we construct a Noetherian domain $\Lambda$ of dimension $1$ that admits an $n\times n$ trace $0$ non-commutator matrix over $\Lambda$ for all $n\ge 2$. We use the following theorem of Heinzer and Levy to construct the ring, and then apply \Cref{cor:d0} with varying maximal ideals to show that there is a trace $0$ non-commutator of arbitrary size.

\begin{theorem}[{\cite[Theorem 2.1]{HL07}}]\label{thm:LH}
  Let $K$ be a field and $I$ a nonempty set. For each $i\in I$, let $(\Lambda_i,\fn(i))$ be a Noetherian local integral domain of dimension $1$ with maximal ideal $\fn(i)$ such that $\Lambda_i$ is a subring of $K$ and the quotient field $Q(\Lambda_i)$ equals $K$. Suppose:
  \begin{enumerate}
  \item Every non-zero element of $K$ is a unit in all but finitely many $\Lambda_i$; and
  \item For every pair of distinct indices $j\ne h$ there exist elements $x_j\in\Lambda_j$ and $x_h\in\Lambda_h$ such that:
    \begin{enumerate}
    \item $x_j$ and $x_h$ are non-units in $\Lambda_j$ and $\Lambda_h$ respectively;
    \item $x_j$ is a unit in $\Lambda_i$ when $i\ne j$, and $x_h$ is a unit in $\Lambda_i$ when $i\ne h$;
    \item $x_j+x_h$ is a unit in  $\Lambda_i$ for every $i$.
    \end{enumerate}
  \end{enumerate}
  Then the ring $\Lambda:=\bigcap_i\Lambda_i$ is Noetherian of dimension $1$, its distinct maximal ideals are $\fm(i):=\fn(i)\cap \Lambda$, and $\Lambda_i=\Lambda_{\fm(i)}$ for each $i\in I$.
\end{theorem}

\unbounded
\begin{proof}
  We will start with the construction of the field $K$ needed to apply \Cref{thm:LH}. Let $k$ be an algebraically closed field and $S:=k[X_{ij}:i\ge 2, 1\le j\le i]$, a polynomial ring with infinitely many variables.
Let $J\subset S$ be the ideal generated by the following elements:
\[ J:= ( X^2_{i1}-X^{p_j}_{ij}: i\ge 2, 2\le j\le i),\] 
where $p_j$ is the $j$-th prime.
\begin{claim}\label{claim:domain}
  $R:=S/J$ is an integral domain.
\end{claim}
\begin{proof}[Proof of \Cref{claim:domain}]
We show that $R$ is an integral domain by proving that $J$ is a prime ideal. So let $g,h\in S$ be such that $gh\in J$. Then 
  \[gh = \sum_{i\ge 2}\sum_{j=2}^i a_{ij}(X^{2}_{i1}-X^{p_j}_{ij}), \]
  for some $a_{ij}\in S$ with only finitely many of the $a_{ij}$'s being non-zero. Then there are only finitely many variables $X_{ij}$ appearing in the equation, so we may restrict our ring $S$ to 
  \[S_n := k[X_{ij}: 2\le i\le n, 2\le j\le i],\] 
  and the ideal $J$ to an ideal
  \[J_n := ( X^{2}_{i1}-X^{p_j}_{ij}: 2\le i\le n, 2\le j\le i)\subset S_n,\]
  for some $n$. Then $g\in J_n$ or $h\in J_n$ will imply $g\in J$ or $h\in J$, so it is sufficient to show that $J_n$ is a prime ideal. We will prove that $S_n/J_n$ is a domain in order to show that $J_n$ is a prime ideal.
  Now,
  \[S_n/J_n \iso k[X_{21},X_{22}]/(X^2_{21}-X_{22}^3) \tensor_k\dots\tensor_k k[X_{n1},\dots,X_{nn}]/(X^2_{n1}-X_{n2}^3,\dots,X^{2}_{n1}-X^{p_n}_{nn}). \]
  A tensor product of finitely generated $k$-algebras that are domains is a domain if $k$ is algebraically closed (see proof of \cite[\href{https://stacks.math.columbia.edu/tag/05P3}{Tag 05P3}]{stacks-project}).
So if \[T_t:=k[X_{1},\dots,X_{t}]/(X^2_{1}-X_{2}^3,\dots,X^{2}_{1}-X^{p_t}_{t}),\] is a domain for all $2\le t\le n$, then $S_n/J_n$ is also a domain.

Now we will prove that $T_t$ is a domain by induction on $t$, along with the fact that $[Q(T_t):k(X_1)]=p_2p_3\cdots p_t$ where $Q(T_t)$ is the quotient field of $T_t$. If $t=2$, then $T_2=k[X_1,X_2]/(X^2_1-X^3_2)$. Now $X_1^2-X_2^3$ is an irreducible polynomial in $k[X_1,X_2]$, and $k[X_1,X_2]$ is a UFD, so $X_1^2-X_2^3$ is a prime element. Hence $T_2$ is a domain, and $[Q(T_2):k(X_1)]=3$. Now suppose that $T_{t-1}$ is a domain and $[Q(T_{t-1}):k(X_1)] = p_2p_3\cdots p_{t-1}$. Since 
\[ T_t \iso T_{t-1}[X_t]/(X_{t}^{p_{t}}-X_1^2),\]
we only need to show that $f_t:=X^{p_t}_t-X_1^2$ is prime in $T_{t-1}[X_t]$ to show that $T_t$ is a domain. We first show that $f_t$ is irreducible in $k(X_1)[X_t]$. Now $f_t=X_t^{p_t}-X_1^2$ has a prime degree and $k(X_1)$ contains all the $p_t$-th roots of unity, so by \cite[Theorem VI.9.1]{Lang}, $f_t$ is irreducible in $k(X_1)[X_{t}]$ if and only if it has no roots in $k(X_1)$. It is clear that $f_t$ has no roots in $k(X_1)$, so $f_t$ is irreducible in $k(X_1)[X_t]$. Now $\deg f_t=p_t$ is a prime and does not divide $[Q(T_{t-1}):k(X_1)]=p_2p_3\cdots p_{t-1}$, so $f_t$ is still irreducible in $Q(T_{t-1})[X_t]$. Hence $f_t$ is a prime in $Q(T_{t-1})[X_t]$ since $Q(T_{t-1})[X_t]$ is a UFD.

Now will use the fact that $f_t$ is a prime in $Q(T_{t-1})[X_t]$ to show that $f_t$ is a prime in $T_{t-1}[X_t]$. So suppose $f_t\mid gh$ for some $g,h\in T_{t-1}[X_t]$. Since $f_t$ is a prime in $Q(T_{t-1})[X_t]$, we may assume that $f_tg' = g$ for some $g':=\sum_{i=0}^{n'}g'_iX_t^i\in Q(T_{t-1})[X_t]$ without loss of generality. Suppose $g'\notin T_{t-1}[X_t]$. Then there exists a maximal index $j$ such that $g'_j\notin T_{t-1}$. If $g=\sum_{i=0}^ng_iX_t^i$, then looking at the degree $j+p_t$ coefficients of the equation $f_tg'=g$, we have
\[ g_{j+p_t} = g'_j - g'_{j+p_t}X_1^2.\]
We have $g'_{j+p_t}\in T_{t-1}$ since $j$ was the maximal index such that $g'_j\notin T_{t-1}$. Hence $g'_j=g_{j+p_t}+g'_{j+p_t}\in T_{t-1}$, which is a contradiction, so $g'\in T_{t-1}[X_t]$. Hence we have $f_t\mid g$ in $T_{t-1}[X_t]$ and so $f_t$ is a prime in $T_{t-1}[X_t]$. This implies that $T_t=T_{t-1}[X_t]/(f_t)$ is a domain, and now $Q(T_t)=Q(T_{t-1})[X_t]/(f_t)$, so 
\[[Q(T_t):k(X_1)] = [Q(T_t):Q(T_{t-1})][Q(T_{t-1}):k(X_1)]= p_tp_2p_3\cdots p_{t-1},\]
and we are done.
Hence $T_t$, $S_n/J_n$ and $R=S/J$ are all domains and we have proved \Cref{claim:domain}.
\end{proof}

Since $R$ is a domain, we may take $K:=Q(R)$, the quotient field of $R$. This will be the $K$ we take for applying \Cref{thm:LH}. We now define the $\Lambda_n$'s needed to apply \Cref{thm:LH}. We take the set $I$ to be $\N$. Let $Y_{ij}\in R\subset K$ be the residue class of $X_{ij}$. For $n\ge 2$, define a subfield
\[F_n := k(Y_{ij}: i\ne n, 1\le j\le i)\subset K,\]
and a subring
\[R_n:= F_n[Y_{n,1},\dots,Y_{n,n}]\subset K.\]
Let $P_n:=F_n[Z_1,\dots,Z_n]$ be the polynomial ring in $n$ variables.
\begin{claim}\label{claim:iso}
  \begin{align*}
  \varphi_n\lmaps[l]{P_n/(Z^{2}_1-Z^3_2,\dots, Z^{2}_1-Z^{p_n}_n)}{R_n=F_n[Y_{n,1},\dots,Y_{n,n}]}\\
  \widebar{Z}_i&\longmapsto Y_i,
\end{align*}
is an isomorphism, where $\widebar{Z}_i$ is the residue class of $Z_i$.
\end{claim}
\begin{proof}[Proof of \Cref{claim:iso}]
It is clear that $\varphi_n$ is surjective, so let us show that it is injective. Consider the lift of $\varphi_n$,
  \begin{align*}
  \tilde{\varphi}_n\lmaps[l]{P_n}{R_n=F_n[Y_{n,1},\dots,Y_{n,n}]}\\
    Z_i&\longmapsto Y_i.
\end{align*}
We see that  $\ker\tilde{\varphi}_n\supset (Z^{2}_1-Z^3_2,\dots, Z^{2}_1-Z^{p_n}_n)$, so to conclude that $\varphi_n$ is injective, we only need to show that 
 \[\ker\tilde{\varphi}_n\subset (Z^{2}_1-Z^3_2,\dots, Z^{2}_1-Z^{p_n}_n).\]
So suppose that $f\in \ker\tilde{\varphi}_n$. If $c\in F^\times_n$ is the product of the denominators of the coefficients of $f$, then 
\[cf\in k[Y_{ij}:i\ne n, 1\le j\le i][Z_1,\dots,Z_n]=:P_n'\subset P_n.\]
Moreover, $cf\in (Z^{2}_1-Z^3_2,\dots, Z^{2}_1-Z^{p_n}_n)$ implies $f\in(Z^{2}_1-Z^3_2,\dots, Z^{2}_1-Z^{p_n}_n)$, so it is sufficient to check that 
\[   \ker\left(\tilde{\varphi}_n|_{P'_n} \right)\subset (Z^{2}_1-Z^3_2,\dots, Z^{2}_1-Z^{p_n}_n)P_n'\subset P_n'.\]
We have the following commutative diagram
\[\begin{tikzcd}
  S=k[X_{ij}:i\ge 2, 1\le j\le i] \arrow[d, "\psi_n", two heads] \arrow[rd, "q", two heads] &   \\
  {P_n'=k[Y_{ij}:i\ne n, 1\le j\le i][Z_1,\dots,Z_n]} \arrow[r, "\tilde{\varphi}_n|_{P'_n}", two heads]   &  \tilde{\varphi}_n(P_n')=k[Y_{ij}:i\ge 2, 1\le j\le i]=S/J
\end{tikzcd},\]
where $q$ is the quotient map and
\begin{align*}
  \psi_n\lmaps[l]{S}{P_n'}\\
  X_{ij} &\longmapsto   \begin{cases}
    Y_{ij} &\text{if }i\ne n,\\
    Z_j &\text{if }i=n.
  \end{cases}
\end{align*}
 Since $\psi_n$ is surjective, $\ker\tilde{\varphi}_n|_{P_n'}=\psi_n(\ker q) = \psi_n(J)$, and 
 \[ \psi_n(J) = ( \psi_n(X_{i1})^2-\psi_n(X_{ij})^{p_j}: i\ge 2, 2\le j\le i) = (Z^{2}_1-Z^3_2,\dots, Z^{2}_1-Z^{p_n}_n) \subset P_n'. \]
 Hence 
 \[ \ker\tilde{\varphi}_n|_{P_n'}= (Z^{2}_1-Z^3_2,\dots, Z^{2}_1-Z^{p_n}_n)\subset P_n',\]
 and $\varphi_n$ is an isomorphism.
\end{proof}
Since $P_n/(Z^{2}_1-Z^3_2,\dots, Z^{2}_1-Z^{p_n}_n)$ is easier to work with than $R_n$, we will use the former presentation to prove properties about $R_n$ to verify the assumptions of \Cref{thm:LH}.

The ideal $(Y_{n,1},\dots,Y_{n,n})\subset R_n$ is maximal since it corresponds to 
$(\widebar{Z}_1,\dots,\widebar{Z}_n)$ so we may localise $R_n$ at $(Y_{n,1},\dots,Y_{n,n})$ to obtain a local domain
\[\Lambda_n := (R_n)_{(Y_{n,1},\dots,Y_{n,n})}\subset K,\]
with a maximal ideal $\fn(n):=(Y_{n,1},\dots,Y_{n,n})\subset \Lambda_n$. Since $\Lambda_n$ is a localisation of a finitely generated $F_n$-algebra, it is Noetherian.

Now we show that $\dim \Lambda_n=1$. We have an isomorphism
\[\Lambda_n\iso \left(P_n/(Z^{2}_1-Z^3_2,\dots, Z^{2}_1-Z^{p_n}_n)\right)_{(\widebar{Z}_1,\dots,\widebar{Z}_n)}, \]
through $\varphi_n$. We may interchange localisation and quotient so
\[\Lambda_n\iso (P_n)_{(Z_1,\dots,Z_n)}/(Z^{2}_1-Z^3_2,\dots, Z^{2}_1-Z^{p_n}_n),\]
as well.
Now
\[\left(Z_1,Z^{2}_1-Z^3_2,\dots, Z^{2}_1-Z^{p_n}_n\right)
  =\left(Z_1, Z_2^3,\dots, Z_n^{p_n}\right) \subset (P_n)_{(Z_1,\dots,Z_n)} \]
 is a $(Z_1,\dots,Z_n)$-primary ideal, so ${Z_1,Z^{2}_1-Z^3_2,\dots, Z^{2}_1-Z^{p_n}_n}$ is a system of parameters of $(P_n)_{(Z_1,\dots,Z_n)}$ (see the start of \cite[Ch. 14]{Matsumura} for the definition of a system of parameters). Hence by \cite[Theorem 14.1]{Matsumura},
\[\dim \Lambda_n=\dim (P_n)_{(Z_1,\dots,Z_n)}/(Z^{2}_1-Z^3_2,\dots, Z^{2}_1-Z^{p_n}_n) = n-(n-1) = 1.\]

Finally note that $\Lambda_n$ contains all $Y_{ij}$'s, so $R\subset \Lambda_n\subset K=Q(R)$. Hence to apply \Cref{thm:LH}, we only need to verify conditions (1) and (2). 

For (1), if $f\in K$, then $f$ can be written as a rational function in the $Y_{ij}$'s. If $f$ is non-zero, then only finitely many $Y_{ij}$'s appear in such a representation of $f$. If $n\ge 2$ is such that $Y_{nj}$ does not appear in such a representation for all $1\le j\le n$, then $f\in F_n$, so $f\in\Lambda_n$ is a unit. Hence any non-zero element in $K$ is a unit in all but finitely many $\Lambda_n$'s.

For (2), if $j \ne h$, then take $x_j:= Y_{j,1}$ and $x_h:=Y_{h,1}$. These are not units in $\Lambda_j$ and $\Lambda_h$ respectively, and are units in $\Lambda_i$ for other $i$'s not equal to $j$ or $h$. Moreover, if $i\ne j$ and $i\ne h$, then $x_j+x_h\in F_i\subset R_i$ so $x_j+x_h$ is a unit in $\Lambda_i$. Hence $\Lambda:=\bigcap_i \Lambda_i$ is a Noetherian domain of dimension $1$ by \Cref{thm:LH}.

Finally, we show that the embedding dimension of $\fm(i):=\fn(i)\cap\Lambda$ is $i$. This will prove the theorem since for all $n\ge 2$, we can find $i$ such that $n\le \tfrac{i+1}{2}$, and we can apply \Cref{cor:d0} with $R=\Lambda$ and $\fm = \fm(i)$ to construct an $n\times n$ trace $0$ non-commutator.

We now compute the embedding dimension of $\fm(i)$'s. Since $\Lambda_i=\Lambda_{\fm(i)}$ by \Cref{thm:LH}, we can check the embedding dimension of the local ring $\Lambda_i$ instead. We have
\[\Lambda_i\iso (P_i)_{(Z_1,\dots,Z_i)}/(Z^{2}_1-Z^3_2,\dots, Z^{2}_1-Z^{p_i}_i),\]
so
\[ \fn(i)/\fn(i)^2\iso (\widebar{Z}_1,\dots,\widebar{Z}_i)/(\widebar{Z}_1,\dots,\widebar{Z}_i)^2.\]
Now we show that $\widebar{Z}_1,\dots,\widebar{Z}_i$ form an $F_i$-basis of $(\widebar{Z}_1,\dots,\widebar{Z}_i)/(\widebar{Z}_1,\dots,\widebar{Z}_i)^2$. So suppose
\begin{equation}
a_1\widebar{Z}_1+\dots+a_i\widebar{Z}_i \in (\widebar{Z}_1,\dots,\widebar{Z}_i)^2,\label{eq:6}
\end{equation}
for some $a_1,\dots,a_i\in F_i$. Now
\[(Z^{2}_1-Z^3_2,\dots, Z^{2}_1-Z^{p_i}_i) \subset (Z_1,\dots,Z_i)^2, \]
so we can lift \cref{eq:6} to $(P_i)_{(Z_1,\dots,Z_n)}$ to obtain
\[ a_1Z_1+\dots+a_iZ_i \in (Z_1,\dots,Z_i)^2.\]
But then $a_j=0$ for all $j$ since all the terms in the left-hand side are degree $1$. Hence 
\[\dim_{F_i}(\widebar{Z}_1,\dots,\widebar{Z}_i)/(\widebar{Z}_1,\dots,\widebar{Z}_i)^2=i \]
\end{proof}

\section{Trace Zero $2\times 2$ Matrices}
\label{sec:2x2}
In this section, we summarise the current progress on the case of $2\times 2$ matrices. We start by recalling various facts and definitions needed to state \Cref{thm:characterisation} where we put together known results from the 1970s and 1980s. We then review facts and definitions on the K-theory of affine surfaces to be able to state new results such as \Cref{cor:fpalgebraisop},\Cref{cor:gradedopring} and \Cref{thm:bbopring}. 

For this section, we assume that $\Spec(R)$ is connected for any ring $R$. 
This is not a strong assumption since we can always reduce the question about commutators to the case where $\Spec(R)$ is connected if $R$ is Noetherian. Indeed, if $R$ is a Noetherian ring, then $\Spec(R)$ has only finitely many connected components, and $R=\prod_{i=1}^r R_i$ where $\Spec(R_i)$ is a connected component of $\Spec(R)$. A matrix $M\in \Mat_n(\prod_{i=1}^rR_i)$ is a commutator if and only if all $M_i$'s are commutators where $M=(M_i)\in \prod_{i=1}^r\Mat_n(R_i)\iso\Mat_n(\prod_{i=1}^rR_i)$.

Lissner showed by elementary means that $\begin{pmatrix}c_1&c_2\\c_3&-c_1 \end{pmatrix}\in \Mat_2(R)$ is a commutator except possibly when $c_i\notin (c_j,c_k)\subset R$ for any $(i\; j\; k)$ which is a permutation of $(1\; 2\; 3)$ (see \cite[Lemma 3.1 and 3.3]{Lissner61}).

The main tool in the $2\times 2$ case is the following lemma connecting a commutator with an exterior power.
Recall that $v\in \bigwedge^pR^n$ is \emph{decomposable} if there exist $v_1,\dots,v_p\in R^n$ such that $v=v_1\wedge \dots \wedge v_p$.
\begin{lemma}\label{lem:2x2decomposable}
  Let $ M:=
  \begin{pmatrix}
    a & b \\
    c & -a
  \end{pmatrix}\in\Mat_2(R)$. Then $M$ is a commutator if and only if $
  ae_2\wedge e_3+be_3\wedge e_1+ce_1\wedge e_2\in\bigwedge^2 R^3$ is decomposable. In particular, every $2\times 2$ trace $0$  matrix is a commutator if and only if every vector in $\bigwedge^2R^3$ is decomposable.
\end{lemma}
\begin{proof}
  Omitted.
\end{proof}

\begin{definition}
  Given  $n\ge 1$, we say that \emph{$R$ is $T_n^{n+1}$} if every vector in $\bigwedge^nR^{n+1}$ is decomposable. We call a ring $R$ an \emph{OP-ring} if it is $T_n^{n+1}$ for all $n\ge 1$.
\end{definition}
This definition was made by David Lissner in \cite{Lissner65} and OP stands for \emph{outer product}.

\begin{remark}
  \Cref{lem:2x2decomposable} can be rephrased as stating that every trace $0$ matrix in $\Mat_2(R)$ is a commutator if and only if $R$ is $T_2^3$. 
\end{remark}

From here on in this section, we will focus on the case when $R$ is Noetherian. For non-Noetherian OP rings, see \cite{JW17}.

Given a property P, we say that a ring $R$ is \emph{locally P} if every localisation of $R$ at a prime ideal has the property P. For $T_n^{n+1}$ or OP, locally true at every prime is equivalent to locally true at every maximal ideal.

\begin{prop}\label{prop:locallyop}
  A Noetherian ring $R$ is locally $T_n^{n+1}$ or locally OP if and only if every localisation of $R$ at a maximal ideal is $T_n^{n+1}$ or OP, respectively.
\end{prop}
\begin{proof}
  If $R$ is locally $T_n^{n+1}$, then every localisation of $R$ at a maximal ideal is $T_n^{n+1}$ by definition.
  On the other hand, suppose every localisation of $R$ at a maximal ideal is $T_n^{n+1}$, and let $\fp\subset R$ be a prime ideal. Then there exists a maximal ideal $\fm\supset \fp$ and $R_\fp$ is a localisation of $R_\fm$ at $\fp R_\fm$. Now let $u\in\bigwedge^nR_\fp^{n+1}$, and $f:R_\fm^{n+1}\to R_\fp^{n+1}$ be the $R_\fm$-module homomorphism induced by the localisation at $\fp R_\fm$. Then there exists $r\in R_\fm\setminus{\fp R_\fm}$ such that $ru \in \im \bigwedge^nf$, so let $v\in \bigwedge^nR_\fm^{n+1}$ be such that $(\bigwedge^{n}f)(v) = ru$. Since $R_\fm$ is $T_n^{n+1}$, $v$ is decomposable, so 
  \[ v = v_1\wedge \dots \wedge v_n \]
  for some $v_1,\dots,v_n\in R_\fm^{n+1}$. Now 
  \[ u = \frac{1}{r}({\textstyle\bigwedge^n}f)(v_1\wedge\dots\wedge v_n)=\frac{1}{r} f(v_1)\wedge \dots\wedge f(v_n)\in {\textstyle\bigwedge^n}R^{n+1}_\fp \]
  so $u$ is decomposable. Hence $R_\fp$ is $T_n^{n+1}$.
  The equivalence for OP follows from the equivalence for $T_n^{n+1}$ for all $n\ge 1$.
\end{proof}

For a local ring, the OP property can be detected by the minimal number of generators for its maximal ideal.
\begin{theorem}[{\cite[Theorem]{Towber70}}]\label{thm:local}
  Let $R$ be a Noetherian local ring with maximal ideal $\fm$. Then $R$ is an OP-ring if and only if $\fm$ is generated by 2 elements.
\end{theorem}

\begin{corollary}\label{cor:regularlocalop}
  Let $R$ be a regular ring of dimension $d \le 2$. Then $R$ is locally OP.
\end{corollary}
\begin{proof}
  Let $\fm\subset R$ be a maximal ideal. Then $R_\fm$ is a regular local ring of dimension $\le d$, so $\fm R_\fm$ is generated by at most 2 elements. Hence $R_\fm$ is an OP ring by \Cref{thm:local}, and so $R$ is locally OP.
\end{proof}

Recall that a \emph{semilocal ring} is a ring with finitely many maximal ideals. For semilocal rings, the OP property can be detected locally.
\begin{theorem}[{\cite[Theorem]{Hinohara72}}]\label{thm:semilocal}
  Let $R$ be a Noetherian semilocal ring. Then $R$ is an OP-ring if and only if $R$ is locally OP.
\end{theorem}

Lissner first used the term OP-ring in \cite{Lissner65} and showed that a Dedekind domain is an OP-ring in \cite[Appendix]{Lissner65}. Towber subsequently showed that if $R$ is a Dedekind domain then $R[x]$ is an OP-ring in \cite[Theorem 1.2]{Towber68}. Estes and Matijevic then proved the following characterisation of OP-rings. Note that in the equivalence stated, instead of the condition locally OP, Estes and Matijevic had $R_\fm$ is OP for any maximal ideal $\fm\subset R$. However, by \Cref{prop:locallyop} that is equivalent to $R$ being locally OP.
Recall that an $R$-module $M$ is \emph{$R$-oriented} if $\bigwedge^nM\iso R$ for some $n\ge 1$.

\begin{theorem}[{\cite[Theorem 1, Corollary 1, Corollary 2]{EM80}}]\label{thm:general}
  Let $R$ be a Noetherian ring satisfying one of the following properties:
  \begin{enumerate}[label=(\alph*)]
  \item $R$ is reduced,
  \item every minimal ideal of $R$ is principal, or
  \item $R$ has only finitely many maximal ideals with non-regular localisation.
  \end{enumerate}
  Then $R$ is an OP-ring if and only if $R$ is locally OP and every finitely generated $R$-oriented module is free.
\end{theorem}
\begin{remark}
  Estes and Matijevic note in \cite[Page 1356]{EM80} that the three conditions on $R$ are probably not necessary, and suspect that every $R$-oriented module being free and $R$ being locally OP are the only necessary conditions for $R$ to be an OP-ring.
\end{remark}

\begin{corollary}\label{cor:orientedisprojective}
  Let $R$ be a reduced Noetherian locally OP ring. Then any finitely generated $R$-oriented module is projective. 
\end{corollary}
\begin{proof}
  Let $M$ be a finitely generated $R$-oriented module, so that $\bigwedge^n M \iso R$ for some $n\ge 1$. We will show that $M$ is projective by showing that it is locally free. So let $\fp\in \Spec(R)$ be a prime ideal. Then
  \[ R_\fp \iso {\textstyle\left(\bigwedge^n M\right)} \tensor_R R_\fp \iso \textstyle{\bigwedge^n} M_\fp.\]
  Now $R$ is locally OP and so $R_\fp$ is an OP-ring, and $M_\fp$ is a finitely generated $R_\fp$-oriented module, so $M_\fp$ is free by \Cref{thm:general}. Hence $M$ is projective.
\end{proof}

Recall that an $R$-module $M$ is \emph{stably-free} if $M\oplus R^m\iso R^n$ for some $n$ and $m$. The \emph{Grothendieck group} of a ring $R$, $K_0(R)$, is the group completion of the monoid of isomorphism classes of finitely generated projective $R$-modules, under direct sum. And $SK_0(R)$ is the subgroup of $K_0(R)$ consisting of the classes $[P]-[R^m]$, where $P$ is an $R$-oriented projective module of constant rank $m$ (cf. \cite[Definition II.2.6.1]{Weibel13}).
\begin{theorem}[Boraty{\'n}ski--Davis--Geramita\cite{BDG80}, Estes--Matijevic\cite{EM80}]\label{thm:characterisation}
  Let $R$ be a Noetherian ring satisfying one of the conditions (a), (b), or (c) in \Cref{thm:general}. Then the following are equivalent:
  \begin{enumerate}
  \item Every trace $0$ matrix in $\Mat_2(R)$ is a commutator,
  \item $R$ is $T_2^3$,
  \item $R$ is an OP-ring,
  \item $R$ is locally OP and every finitely generated projective $R$-oriented module is free,
  \item $R$ is locally OP, $SK_0(R)=0$ and every finitely generated stably-free projective $R$-module is free, and 
  \item Every maximal ideal of $R$ is generated by two elements and every finitely generated stably-free projective $R$-module is free.
  \end{enumerate}
\end{theorem}
\begin{proof}
  (1) is equivalent to (2) by \Cref{lem:2x2decomposable}.
  (3) implies (2) follows from the definition of an OP-ring. (4) implies (3) is \Cref{thm:general} with \Cref{cor:orientedisprojective}.

  For (5) implies (4), suppose $M$ is an $R$-oriented module. Then for all prime ideal $\fp$, $M_\fp$ is $R_\fp$-oriented module, and since $R_\fp$ is OP, every $R_\fp$-oriented module is free. Hence $M$ is an $R$-oriented projective module. Hence it is stably-free projective by \cite[(6) Lemma]{BDG80} and hence free by (5).

  For (6) implies (5), if every maximal ideal of $R$ is generated by 2 elements, then every maximal ideal is locally generated by at most two elements. Hence every localisation is OP by \Cref{thm:local}. Now for all maximal ideal $\fm$, $\fm R_\fm$ is generated by at most 2 elements, so $\dim R_\fm\le 2$ by Krull's Height theorem \cite[Theorem 13.5]{Matsumura}. Hence $\dim R\le 2$, and so by \cite[(5) Proposition]{BDG80}, $SK_0(R) = 0$.

(2) implies (6) follows from \cite[(3) Lemma, (4) Lemma]{BDG80}.
\end{proof}

\begin{corollary}\label{cor:lowdim}
  Let $R$ be a locally OP Noetherian ring. Then
  \begin{enumerate}
  \item $\dim R\le 2$,
  \item If $\dim R=0$ then $R$ is an OP-ring,
  \item If $\dim R=1$ and if $R$ satisfies one of the conditions (a), (b), or (c) in \Cref{thm:general}, then $R$ is an OP-ring.
  \end{enumerate}
\end{corollary}
\begin{proof}
  If $R$ is locally OP Noetherian ring, then by \Cref{thm:local}, $\fm R_\fm$ is generated by 2 elements for any maximal ideal $\fm$ of $R$. Hence $\dim R_\fm \le 2$ by Krull's Height theorem \cite[Theorem 13.5]{Matsumura}, and so $\dim R\le 2$. 

  (2) follows from \Cref{thm:semilocal} since if $R$ is Noetherian and dimension $0$, then it is a semilocal ring.

  For (3), let $M$ be an $R$-oriented module, that is $\bigwedge^nM=R$ for some $n\ge 1$. Then by \Cref{cor:orientedisprojective} $M$ is projective 
and so by Bass cancellation theorem \cite[Theorem I.2.3]{Weibel13}, $M\iso P\oplus R^{n-1}$ for some projective module $P$ of rank 1. Now
 \[R = {\textstyle\bigwedge^nM= \bigwedge^n}\left(P\oplus R^{n-1}\right)= \bigoplus_{k=0}^n {\textstyle\bigwedge^kP} \tensor {\textstyle\bigwedge^{n-k}R^{n-1}} = P\tensor R = P,\]
 so $M$ is free. Hence every finitely generated $R$-oriented module is free, so $R$ is an OP-ring by \Cref{thm:general}.
\end{proof}

\begin{example}
  Let $A$ be a PID with quotient field $F$ and $K/F$ be a finite field extension. An $A$-order $R$ in $K$ is an $A$-subalgebra of $K$ which is finitely generated as an $A$-module and such that $F\tensor_AR=K$. In \cite[Theorem 3.4]{Clark18}, Clark showed that if $N:=[K:F]$, then there exists an $A$-order $R$ in $K$ such that $R$ admits a maximal ideal $\fm\subset R$ with $N$ as the minimal cardinality of a set of generators. If $N\ge 3$, then by \Cref{thm:characterisation} (6), $R$ is not an OP ring. In particular if $K/\Q$ is a number field of degree at least $3$, then $K$ admits a $\Z$-order which is not an OP ring.

  Note that if $R$ is an $A$-order in a quadratic extension $K/F$, then every ideal of $R$ is a free $A$-module of rank $2$, so every maximal ideal of $R$ is generated by at most $2$ elements. Hence for any maximal ideal $\fm\subset R$, $\fm R_\fm$ is generated by at most $2$ elements, so by \Cref{thm:local}, $R$ is locally OP. Moreover, $R$ is also a dimension $1$ domain, so by \Cref{cor:lowdim} (3)(a), $R$ is an OP ring.
\end{example}

From here on, we will work with the case where the ring $R$ is a $k$-algebra for some field $k$. As we will see later, this will allow us to utilise the geometry of $\Spec(R)$ as a $k$-variety.
The next theorem gives examples of $k$-algebras $R$ where every stably-free projective module is free.
\begin{theorem}\label{thm:stablyfreeisfree}
  Let $k$ be a ring, and $R$ be a finitely generated $2$-dimensional $k$-algebra. Then every finitely generated stably-free projective $R$-module is free if any of the following is satisfied:
  \begin{enumerate}
  \item $k$ is an algebraically closed field \cite[Theorem 1]{MS76}.
  \item $k$ is an infinite perfect field with $\ch k\ne 2$ and the cohomological dimension of $k$ is at most $1$ \cite[Remark 4.2]{Bhatwadekar03}.
  \item $k$ is a real closed field and all $k$-points on $\Spec(R)$ lie on a closed subscheme of dimension $\le 1$ \cite[Theorem 3.1]{MS76}.
  \item $k=\Z$ or $k=\F_q$ for any prime power $q$ \cite[Corollary 2.5]{KMR88}.
  \end{enumerate}
\end{theorem}

\begin{corollary}\label{cor:generatedby2}
  Let $k$ be a ring and $R$ be a Noetherian reduced finitely generated $2$-dimensional $k$-algebra satisfying one of the conditions (1)-(4) in \Cref{thm:stablyfreeisfree}. Then every trace $0$ matrix in $\Mat_2(R)$ is a commutator if and only if every maximal ideal of $R$ can be generated by two elements.
\end{corollary}
\begin{proof}
  Follows from applying \Cref{thm:stablyfreeisfree} to the equivalence of (1) and (6) in \Cref{thm:general}.
\end{proof}

We now state an application of \Cref{thm:stablyfreeisfree} (1) which will be used in \Cref{ex:hypersurface} and \Cref{ex:godeaux}.
Recall that for a variety $X/k$, the Chow group $A_0(X)$ is the group of zero cycles of degree $0$ modulo rational equivalence.
For a projective variety $X/k$, we have the \emph{Albanese map}
\[ AJ_X \maps{A_0(X)}{\Alb_{X/k}(k)}, \]
where $\Alb_{X/k}$ is the Albanese variety of $X/k$ (see \cite[Section 3]{ss03} for more details about the Albanese map). We denote the kernel as $SA_0(X) := \ker AJ_X$, which is also denoted as $T(X)$ or $F^2CH_0(X)$ in some literature.
\begin{theorem}\label{cor:sa}
  Let $k$ be an algebraically closed field and suppose that $R$ is a $2$-dimensional regular domain that is a finitely generated $k$-algebra. If $\Spec(R)$ is an open affine subvariety of a regular projective surface $X/k$ with finite $SA_0(X)$, then $R$ is an OP-ring, and every trace $0$ matrix in $\Mat_2(R)$ is a commutator.
\end{theorem}
\begin{proof}
  By \Cref{cor:regularlocalop}, $R$ is a locally OP ring, and by \cite[Theorem 3]{MS76}, $SA_0(X)$ finite implies $SK_0(R)=0$. Every stably-free projective $R$-module is free by  \Cref{thm:stablyfreeisfree} (1), so $R$ satisfies \Cref{thm:characterisation} (5). Hence $R$ is an OP-ring and every trace $0$ matrix in $\Mat_2(R)$ is a commutator.
\end{proof}

When $k=\widebar{\F}_p$,  we have the following theorem.
\begin{theorem}\label{cor:positivecharop}
  Let $R$ be a locally OP dimension $2$ finitely generated $\widebar{\F}_p$-algebra that satisfies one of the conditions (a), (b), or (c) in \Cref{thm:general}. Then $R$ is an OP-ring and every trace $0$ matrix in $\Mat_2(R)$ is a commutator.
\end{theorem}
\begin{proof}
  By \Cref{thm:general}, we only need to show that every finitely generated oriented $R$-module is free. Let $P$ be an oriented module of rank $n$. If $n >2$ then by Bass cancellation theorem \cite[Theorem I.2.3]{Weibel13}, $P\iso Q\oplus R^{n-2}$ where $Q$ is a projective module of rank $2$. By \cite[Theorem 6.4.1]{KS07}, any projective $R$-module of rank $2$ has a non-zero free direct summand, so $Q\iso Q'\oplus R$ for some finitely generated projective $R$-module $Q'$. Now
 \[R = {\textstyle\bigwedge^nP= \bigwedge^n}\left(Q'\oplus R^{n-1}\right)= \bigoplus_{k=0}^n {\textstyle\bigwedge^k}Q' \tensor {\textstyle\bigwedge^{n-k}}R^{n-1} = Q'\tensor R = Q'.\]
 Hence $P$ is free.
\end{proof}

\fpalgebra
\begin{proof}
  $R$ is locally OP by \Cref{cor:regularlocalop}, and $R$ is regular so it is reduced. Hence $R$ is an OP ring by \Cref{cor:positivecharop}, and so every trace $0$ matrix in $\Mat_2(R)$ is a commutator.
\end{proof}

For an algebra over a characteristic $0$ field, we have the following result for a graded $\widebar{\Q}$-algebra.
\begin{theorem}\label{cor:gradedopring}
Let $R=\oplus_{n\ge 0}R_n$ be a $2$-dimensional regular graded domain that is an associative algebra over $R_0=\widebar{\Q}$. Then $R$ is an OP ring, and every trace $0$ matrix in $\Mat_2(R)$ is a commutator.
\end{theorem}
\begin{proof}
  Since $R$ is a regular $2$-dimensional ring, $R$ is locally OP by \Cref{cor:regularlocalop}. By \cite[Theorem 1.2]{KS02}, every finitely generated projective $R$-module is free, so every $R$-oriented module is free. Hence $R$ is an OP ring by \Cref{thm:general}, and every trace $0$ matrix in $\Mat_2(R)$ is a commutator.
\end{proof}

Recall the following conjecture (see \cite[Page 267 and Theorem 6.2.1]{KS07} for a discussion about the conjecture).
\begin{conjecture}[Bloch--Beilinson Conjecture]\label{bbconjecture}
  If $X/\widebar{\Q}$ is an irreducible regular projective surface, then $SA_0(X)=0$.
\end{conjecture}

\begin{theorem}\label{thm:bbopring}
  Assume the Bloch--Beilinson conjecture. Then every finitely generated dimension $2$ regular $\widebar{\Q}$-algebra $R$ is an OP-ring. In particular, every trace $0$ matrix in $\Mat_2(R)$ is a commutator.
\end{theorem}
\begin{proof}
  For any $R$ as in the theorem, there exists an irreducible regular projective surface $X/\widebar{\Q}$ birational to $\Spec(R)$. By \Cref{bbconjecture}, $SA_0(X)=0$, so by \cite[Theorem 3]{MS76}, $SK_0(R)=0$. By \Cref{thm:stablyfreeisfree} (1), every stably-free projective module is free. Finally,  $R$ is locally OP since $R$ is regular and dimension 2. Hence $R$ satisfies \Cref{thm:characterisation} (5) so $R$ is an OP-ring, and every trace $0$ matrix in $\Mat_2(R)$ is a commutator.
\end{proof}

In contrast to the case of $\widebar{\Q}$-algebras, the following theorem gives examples of dimension $2$ regular $\C$-algebras which are not OP-rings.
\begin{theorem}\label{cor:infchownotop}
  Let $X/\C$ be an irreducible regular proper surface, with $H^0(X, \Omega^2_{X/\C})\ne 0$, and let $\Spec(R)$ be a non-empty open affine subvariety of $X$. Then $R$ is not an OP ring and there is a trace $0$ non-commutator in $\Mat_2(R)$.
\end{theorem}
\begin{proof}
  By \cite[Corollary 1]{KS10}, $A_0(\Spec R)\ne 0$. Since $\Spec(R)$ is a regular affine surface, $A_0(\Spec R)=SK_0(R)$ (see \cite[Theorem 4.2 (d)]{MS76}) and so it does not satisfy \Cref{thm:characterisation} (5). Hence $R$ is not an OP ring and there exists a trace $0$ non-commutator in $\Mat_2(R)$.
\end{proof}

\begin{example}\label{ex:hypersurface}
  For any $d\ge 1$, let $R_d:=\C[x,y,z]/(x^d+y^d+z^d-1)$. Then $\Spec(R_d)$ is an open subvariety of $X_d:= \Proj(\C[x_0,x_1,x_2,x_3]/(x_0^d+x_1^d+x_2^d+x_3^d))$. 
  If $d=1,2,3$, $X_d$ is rational \cite[Example II.8.20.3]{hartshorne}, and so $A_0(X_d)=0$ \cite[Prop. 7.1]{bloch10}. Hence $SA_0(X_d)=0$ and by \Cref{cor:sa}, $R_d$ is an OP-ring and every trace $0$ matrix in $\Mat_2(R_d)$ is a commutator.

If $d\ge 4$, then we have $\Omega^2_{X_d/\C}=\cO_{X_d}(d-4)$ (see \cite[Example II.8.20.3]{hartshorne}), so $H^0(X_d, \Omega^2_{X_d/\C}) = H^0(X_d,\cO_{X_d}(d-4))\ne 0$. Hence by \Cref{cor:infchownotop}, there is a trace $0$ non-commutator in $\Mat_2(R_d)$. Note that if  the Bloch-Beilinson conjecture is true, then $\widebar{\Q}[x,y,z]/(x^d+y^d+z^d-1)$ is an OP-ring by \Cref{thm:bbopring} even if $d\ge 4$.
\end{example}
Recall the following conjecture of Bloch (see Conjecture 1.8 and Proposition 1.11 in \cite{bloch10} for details about the conjecture).
\begin{conj}[Bloch Conjecture]\label{conj:bloch}
  Let $X/\C$ be a regular projective surface. If $H^0(X,\Omega^2_{X/\C})=0$ then $SA_0(X)=0$.
\end{conj}

Bloch's conjecture has been verified in many cases, including for any surfaces that are not of general type \cite[Proposition 4]{BKL76} and for some surfaces of general type, see e.g. \cite{bcgp12,IM79,Barlow85rational,Voisin14,Bauer14,BF15,PW16}. Thus in these cases, we can apply \Cref{cor:sa} to obtain examples of OP-rings. For example, the following is an OP-ring coming from a surface of general type called a Godeaux surface.
\begin{example}\label{ex:godeaux}
  Let $Y$ be the quintic complex surface in $\P^3_\C$ defined by the equation $x_0^5+x_1^5+x_2^5+x_3^5=0$. Let 
  \[\sigma(x_0:x_1:x_2:x_3) = (x_0: \zeta_5x_1:\zeta_5^2x_2:\zeta_5^3x_3)\]
  be an automorphism of $Y$ where $\zeta_5=e^{2\pi i/5}$. The quotient surface $X:=Y/\langle\sigma\rangle$ is called a Godeaux surface, with $A_0(X)=0$ (see \cite[Theorem 1]{IM79}).

So for an open affine subvariety $\Spec(R)$ of $X$, we have that $R$ is an OP-ring by \Cref{cor:sa}. For example consider
\[S:= \C[x,y,z]/(x^5+y^5+z^5+1),\]
where $x:=x_1/x_0$, $y:=x_2/x_0$ and $z:=x_3/x_0$, so that $\Spec(S)$ is an open subscheme of $Y$. Then $\sigma$ acts on $\Spec(S)$ as well by $\sigma(f(x,y,z))= f(\zeta_5 x, \zeta_5^2 y, \zeta_5^3 z)$, so $\Spec(S^{\langle \sigma\rangle}) \iso \Spec(S)/\langle\sigma\rangle$ is an open affine subvariety of $X$. Hence for
\[R:=\left(\C[x,y,z]/(x^5+y^5+z^5+1) \right)^{\langle\sigma\rangle}, \]
every trace $0$ matrix in $\Mat_2(R)$ is a commutator.
\end{example}

We conclude this section with the case when $\Spec(R)$ is an affine quadric hypersurface in $\A^3_k$. We do not assume that $k$ is algebraically closed, and use the theory of quadratics forms to determine whether $R$ is an OP ring.

Given $k$ a field with $\ch k\ne 2$ and a homogeneous degree $2$ polynomial $q(x,y,z)\in k[x,y,z]$, define 
\[R(k,q):=k[x,y,z]/(q-1).\]
Recall that $q$ can be written as $q=(x \; y\; z)Q(x\; y\; z)^t$ where $Q\in\Mat_3(k)$ is a symmetric matrix. The \emph{discriminant} of $q$ is $\Delta(q):=\det(Q)$ and $q$ is called \emph{non-degenerate} if $\Delta(q)\ne 0$. If $q$ is non-degenerate, then $R(k,q)$ is a regular $2$-dimensional algebra. Finally, recall that $q$ is \emph{isotropic over $k$} if there exist $x,y,z\in k$ not all $0$ such that $q(x,y,z)=0$.
\begin{theorem}\label{cor:quadratic1}
  If $q$ is isotropic or $\sqrt{-\Delta(q)}\in k$, then $R(k,q)$ is an OP-ring. Hence every trace $0$ matrix in $\Mat_2(R(k,q))$ is a commutator.
\end{theorem}
\begin{proof}
  Let $R:=R(k,q)$. Since $R$ is a regular dimension $2$ ring, it is a locally OP ring by \Cref{cor:regularlocalop}. Now suppose $P$ is an $R$-oriented projective module with $\bigwedge^n P =R$. Then by \cite[Theorem 16.1]{swan87}, $P=R^{n-1}\oplus Q$ with $\rk Q=1$. So
  \[R={\textstyle\bigwedge^n P = \bigwedge^n}(R^{n-1}\oplus Q) = R\tensor Q = Q.\]
  Hence $P$ is a free module, and by \Cref{thm:characterisation}, $R$ is an OP ring and every trace $0$ matrix in $\Mat_2(R)$ is a commutator.
\end{proof}

We also have a partial converse of \Cref{cor:quadratic1}.
\begin{theorem}\label{cor:quadratic2}
  If $q$ is anisotropic, $\sqrt{-\Delta(q)}\notin k$ and $q$ represents 1, then $R(k,q)$ is not an OP-ring and there is a trace $0$ matrix in $\Mat_2(R(k,q))$ that is not a commutator.
\end{theorem}
\begin{proof}
  Let $R:=R(k,q)$. Since 
  \[SK_0(R) = \ker(\tilde{K}_0(R)\to \Pic(R)),\] 
  $SK_0(R)=\Z/2\Z$ by \cite[Theorem 9.2 (b), Lemma 11.5]{swan87}. Hence by \Cref{thm:characterisation}, $R$ is not an OP ring and there is a trace $0$ matrix in $\Mat_2(R)$ that is not a commutator.
\end{proof}

\begin{remark}
  Over $R:=\R[x,y,z]/(x^2+y^2+z^2-1)$, there is a well-known example of a $2\times 2$ non-commutator $A:=\begin{pmatrix}
    x & y \\ z & -x
  \end{pmatrix}\in \Mat_2(R)$ (see \cite[Section 3]{RR00} for the proof). By taking $q:=x^2+y^2+z^2$ as the quadratic form, we see that $R=R(\R, q)$. Since $\sqrt{-\Delta(q)}=\sqrt{-1}\notin \R$, \Cref{cor:quadratic2} implies that there is a non-commutator in $\Mat_2(R)$. However we cannot conclude from \Cref{cor:quadratic2} that this particular matrix $A$ is a non-commutator. 

Note that for $R\tensor_\R\C = R(\C, q)= \C[x,y,z]/(x^2+y^2+z^2-1)$, \Cref{cor:quadratic1} applies since $\sqrt{-\Delta(q)}=\sqrt{-1}\in \C$, so $A$ is a commutator in $\Mat_2(R(\C,q))$. For example, we can write $A$ as
  \[
    A = \left[ \begin{pmatrix}
      1+ix(ix-y) & -xz \\
      x(ix-y) & 0 
    \end{pmatrix},
    \begin{pmatrix}
      -iz & ix+y\\
      -z  & 0
    \end{pmatrix} \right].
\]
\end{remark}

\section*{Acknowledgements}
The author would like to thank his advisor, Dino Lorenzini, for posing the problem and providing detailed feedback on the manuscript. This work was completed as a part of the author's doctoral dissertation at University of Georgia.

\printbibliography
\end{document}